\documentclass[preprint,sort&compress,12pt,draft]{elsarticle}
\usepackage{amsmath}
\usepackage{mathrsfs}
\usepackage{stmaryrd}
\usepackage{bm}
\usepackage{amsfonts}
\usepackage{amssymb}
\usepackage{amsthm}
\usepackage[RGB]{xcolor}

\newcommand{\be}{\begin{equation}}
\newcommand{\bea}{\begin{equation}\begin{aligned}}
\newcommand{\beas}{\begin{equation*}\begin{aligned}}
\newcommand{\eeas}{\end{aligned}\end{equation*}}
\newcommand{\eea}{\end{aligned}\end{equation}}
\newcommand{\ee}{\end{equation}}

\makeatletter
\def\ps@pprintTitle{%
	\let\@evenfoot\@oddfoot
}
\makeatother
\begin{document}
\begin{frontmatter}
\title{Global well-posedness for 3D incompressible magneto-micropolar fluids without resistivity
and spin viscosity in strip domains}

\author[sJ]{Youyi Zhao}
\ead{zhaoyouyi957@163.com}
\address[sJ]{School of Mathematics and Statistics, Fuzhou University, Fuzhou, 350108, China.}
\begin{abstract}
The global existence of classical solutions to the 3D compressible magneto-micropolar fluid system without resistivity and spin viscosity in a strip domain was recently established by Feng, Hong, and Zhu [Sci. China Math., 2024].
While Lin and Xiang [Sci. China Math., 2020] established global well-posedness for the 2D incompressible counterpart,
the global well-posedness for the 3D incompressible case remains open.
The analysis is rendered difficult by three major obstacles which are further compounded in the 3D case:
the degeneracy induced by the lack of magnetic diffusion and spin viscosity; the coupling between micro-rotation and velocity fields characterized by a non-dissipative anti-symmetric structure; and the interaction between the velocity, magnetic field, and pressure, where the pressure acts as a non-state variable.
In this paper, by adapting the two-layer energy method of Guo and Tice [Arch. Ration. Mech. Anal., 2013] and the techniques employed in Feng et al., together with refined trace estimates, we overcome these difficulties and establish the global well-posedness of classical solutions to the 3D incompressible magneto-micropolar fluid system without resistivity and spin viscosity in a strip domain. Moreover, we demonstrate the algebraic time-decay of solutions toward the equilibrium.
\end{abstract}

\begin{keyword}
Magneto-micropolar fluids; Incompressible; 3D; Global well-posedness; Time decay.
\end{keyword}
\end{frontmatter}


\newtheorem{thm}{Theorem}[section]
\newtheorem{lem}{Lemma}[section]
\newtheorem{pro}{Proposition}[section]
\newtheorem{cor}{Corollary}[section]
\newproof{pf}{\emph{Proof}}
\newdefinition{rem}{Remark}[section]
\newtheorem{definition}{Definition}[section]

\section{Introduction}\label{sec:01}
\numberwithin{equation}{section}

\subsection{Background and motivation}\label{0608241825n}
The micropolar fluid model, first proposed by Eringen \cite{Eringen1966}, has been extensively studied both in the engineering and mathematical literature.
This model describes a class of fluids exhibiting microscopic effects arising from the local structure and micromotions of the fluid particles.
Physically, micropolar fluids may represent media
consisting of the motion of rigid, randomly oriented or spherical particles that have their own spins or microrotations suspended in a viscous medium.
Typical examples include liquid crystals, ferrofluids, colloidal suspensions, and biological fluids such as blood.
Mathematically, this model introduces a micro-rotation field alongside the standard velocity and pressure, and is accordingly governed by an additional conservation law for angular momentum \cite{Lukaszewicz1999}.
This results in a system with microstructure and non-symmetric stress tensors, fundamentally distinguishing it from standard viscous theories.
More background information on the micropolar fluid model can be found in \cite{Cowin1968,Erdogan1970,Eringen1969} and the references therein.

When considering flows such as electrically conducting fluids in a magnetic field or polarized fluids in an electric field, electromagnetic effects must be included. By coupling the conservation laws for micropolar fluids with the equations of electrodynamics, one comes to the magnetohydrodynamics of micropolar fluids \cite{ahmadi1974universal}.
The magneto-micropolar fluid system is typically employed to describe
the motion of aggregates of small solid ferromagnetic particles in viscous magnetic fluids (such as ferrofluids, where magnetic nanoparticles are suspended in base liquids like esters or fluorocarbons), which is of great importance in both practical and mathematical applications \cite{Ball1996}.
Mathematically, the governing system for the incompressible magneto-micropolar fluids is given as follows:
\begin{equation}\label{010100a}
\begin{cases}
\tilde{u}_t + \tilde{u} \cdot \nabla \tilde{u} + \nabla (\tilde{p} +\lambda|\tilde{B}|^2/2)-(\mu + \chi) \Delta \tilde{u} =\lambda \tilde{B}\cdot\nabla \tilde{B}
+2\chi \nabla \times \tilde{w}, \\[0.5mm]
\tilde{B}_t + \tilde{u} \cdot \nabla \tilde{B}=\nu\Delta \tilde{B}+ \tilde{B}\cdot \nabla \tilde{u}, \\[0.5mm]
\tilde{w}_t + \tilde{u} \cdot \nabla \tilde{w} + 4\chi \tilde{w} = 2\chi \nabla \times \tilde{u} {+ \gamma \Delta \tilde{w}  + \kappa \nabla \mathrm{div} \tilde{w}}, \\[0.5mm]
\mathrm{div}\tilde{u}=\mathrm{div}\tilde{B}=0.
\end{cases}
\end{equation}
Here the unknowns $\tilde{u}:= (\tilde{u}_1,\tilde{u}_2,\tilde{u}_3)^\top(x,t)$, $\tilde{B}:= (\tilde{B}_1,\tilde{B}_2,\tilde{B}_3)^\top(x,t)$, $\tilde{w}:= (\tilde{w}_1,\tilde{w}_2, \tilde{w}_3)^\top(x,t)$ and $\tilde{p}:=\tilde{p}(x,t)$ represent the velocity, magnetic field, micro-rotation velocity, and pressure, respectively.
The constants $\mu$, $\chi$, $\lambda$ and $\nu$ stand for the kinematic viscosity, micro-rotation viscosity, permeability of vacuum dividing by $4\pi$, and magnetic
diffusion coefficient, respectively. Additionally, the constants $\gamma$ and $\kappa$ represent the spin viscosities.
In the governing system \eqref{010100a},
we refer to \eqref{010100a}$_1$ as the momentum equations,  \eqref{010100a}$_2$ as the magnetic induction equations, and \eqref{010100a}$_3$ as the micro-rotation equations.
Since the fluid is incompressible and the magnetic field $B$ is source free, we naturally pose the condition \eqref{010100a}$_4$.

It should be noted that the micro-rotation viscosity coefficient $\chi>0$ in \eqref{010100a} is crucial for the micropolar fluid model; otherwise, the velocity and micro-rotation decouple, and the global motion remains unaffected by the micro-rotation.
Due to their relevance in both mathematical and physical contexts, the mathematical theory of magneto-micropolar fluids has been extensively studied.
For the system \eqref{010100a} with full viscosities, the existence, uniqueness, and regularity of weak solutions were established via the standard energy method \cite{RojasMedar1998}. The local existence of strong solutions for arbitrary initial data was proved in \cite{RojasMedar1997}, while the global existence of strong solutions with small initial data was analyzed in \cite{OrtegaTorres1999}.
Beyond the case of full viscosities, significant attention has recently been devoted to the mathematical analysis of system \eqref{010100a} with partial (or mixed) dissipation.
Cheng--Liu \cite{ChengLiu2015} proved the global regularity of the 2D anisotropic magneto-micropolar fluid system with vertical kinematic viscosity, horizontal magnetic diffusion, and horizontal spin viscosity. Exploiting the structure of the system, Yamazaki \cite{Yamazaki2015} established global regularity for solutions in the absence of spin viscosity in 2D. Later, Regmi--Wu \cite{RegmiWu2016} studied the global existence and regularity of classical solutions to the 2D incompressible magneto-micropolar equations with partial dissipation, singling out three special cases of it. For cases involving other mixed partial viscosities, the global existence and regularity were successively established by Ma \cite{Ma2018} and Lin et al. \cite{LinLiuZhangSun2024}, among others. It is worth noting that most of these results are confined to the 2D setting.

In the absence of both magnetic diffusion and spin viscosities (i.e., $\nu=\gamma=\kappa=0$), the governing system reduces to
\begin{equation}\label{0101}
\begin{cases}
\tilde{u}_t + \tilde{u} \cdot \nabla \tilde{u} + \nabla (\tilde{p} +\lambda|\tilde{B}|^2/2)-(\mu + \chi) \Delta \tilde{u} =\lambda \tilde{B}\cdot\nabla \tilde{B}
+2\chi \nabla \times \tilde{w}, \\[0.5mm]
\tilde{B}_t + \tilde{u} \cdot \nabla \tilde{B}= \tilde{B}\cdot \nabla \tilde{u}, \\[0.5mm]
\tilde{w}_t + \tilde{u} \cdot \nabla \tilde{w} + 4\chi \tilde{w} = 2\chi \nabla \times \tilde{u} , \\[0.5mm]
\mathrm{div}\tilde{u}=\mathrm{div}\tilde{B}=0.
\end{cases}
\end{equation}
The analysis of this system is mathematically delicate due to the absence of dissipative mechanisms in the magnetic induction and micro-rotation equations.
Moreover, the system exhibits strong coupling between the micro-rotation and velocity fields characterized by a non-dissipative anti-symmetric structure.

In the absence of the magnetic field (i.e., $\tilde{B}=0$),
system \eqref{0101} reduces to the incompressible micropolar fluid system without spin viscosity. Dong--Zhang \cite{8} proved  the global existence and
uniqueness of smooth solutions in $\mathbb{R}^2$
by introducing a novel auxiliary quantity; however, the corresponding 3D case and the 2D/3D initial-boundary value problems remain open.
It is worth noting that magnetic fields generally possess a stabilizing effect on fluid dynamics. In the context of linearized non-resistive MHD, Chandrasekhar first demonstrated that magnetic fields can inhibit thermal instability in electrically conducting fluids \cite{CSHHSCPO}. Subsequent numerical studies further suggested that energy in an inhomogeneous plasma dissipates at a rate independent of resistivity \cite{CFCCRI}. This implies that the nonlinear non-resistive MHD system may still exhibit dissipative behavior induced by the magnetic field, thereby allowing for global solutions, at least for small initial data. Such stability has been mathematically verified in existing literature; see, for instance, \cite{BCSCSPLL, JFJSJMFMOSERT, RXXWJHXZYZZF, TZWYJGw,WYTIVNMI,JFJSOUI}. Additionally, the magnetic inhibition theory \cite{JFJSOMITIN} provides a physical and mathematical understanding of this stabilization, grounded in the frozen-in flux theorem (magnetic flux conservation) and the principle of minimum potential energy.

For the magneto-micropolar fluid system \eqref{0101} in a 2D strip domain, Lin--Xiang \cite{lin2020global} established the global well-posedness of strong solutions subject to a uniform horizontal magnetic field $\bar{B}:=(1,0)$.
Their analysis relied on the divergence-free condition $\mathrm{div}\tilde{B}=0$ and the Navier-slip boundary condition for the velocity $\tilde{u}$.
However, it is non-trivial to apply such an approach to the 3D case.
Leveraging the stabilizing effect of magnetic fields, Feng--Hong--Zhu \cite{FengSCM} very recently proved the global existence of classical solutions
for the 3D compressible counterpart under a uniform vertical magnetic field, formulated in Lagrangian coordinates
(see also Zhai--Wu--Xu \cite{zhai2025stability} for the study of the periodic domain case under special condition).
The analysis in Feng et al. hinges on a novel div-curl decomposition technique for the micro-rotation velocity
alongside the coupling between the fluid velocity and the micro-rotation field.
However, a crucial ingredient in their approach is the presence of a density-dependent pressure term.
Despite these advances, the global well-posedness of the 3D incompressible magneto-micropolar fluid system (i.e., system \eqref{0101}) in a strip domain has not yet been addressed.
Compared with the work in \cite{FengSCM}, the primary difficulty in the incompressible case lies in the fact that the pressure becomes an unknown function rather than a known function of density. This poses a fundamental obstacle to extending the analysis of \cite{FengSCM} to the incompressible regime, as noted therein.
Furthermore, the presence of physical boundaries introduces additional technical difficulties that complicate the analysis.
Motivated by this, in this paper, we aim to establish the global well-posedness of solutions for such a system in a 3D strip domain. More precisely, we consider system \eqref{0101} in the strip domain $\Omega:=\mathbb{R}^2\times(0,1)$ subject to the following initial-boundary value conditions:
\begin{align}
&\label{0101b}
(\tilde{u},\tilde{B},\tilde{w})\big|_{t=0}=(\tilde{u}^0,\tilde{B}^0,\tilde{w}^0)\quad\mbox{in}\;\Omega,\\[1mm]
&\label{0101c}
\tilde{u}\big|_{\partial\Omega}=0\quad\mbox{for any}\;t>0,
\end{align}
where $\partial\Omega$ denotes the boundary of $\Omega$, namely, $\partial\Omega:=\mathbb{R}^2\times\{0,1\}$.
We mention that the magnetic induction equations and micro-rotation equations are hyperbolic and characteristic,
and hence no boundary condition needs to be imposed for the magnetic field and micro-rotation velocity.

\subsection{Reformulation in Lagrangian coordinates}\label{0608241825}

Similarly to \cite{TZWYJGw,FengSCM,WYTIVNMI,JFJSJMFMOSERT,JFJSOUI},
it is more convenient and effective to work with Lagrangian coordinates.
To this end, let $(\tilde{u},\tilde{B},\tilde{w})$ be the solution of the initial-boundary value problem \eqref{0101}--\eqref{0101c}.
We use the flow map $\zeta$ to define the transformation matrix $\mathcal{A}:=(\mathcal{A}_{ij})_{3\times 3}$
via $\mathcal{A}^{\top}:=(\partial_j \zeta_i)^{-1}_{3\times 3}$,
where $\zeta$ is the solution to the following initial value problem
\begin{equation}
\label{201806122101}
\quad\begin{cases}
\partial_t \zeta(y,t)=\tilde{u}(\zeta(y,t),t)&\mbox{ in }\Omega\times\mathbb{R}^{+},\\[0.5mm]
\zeta(y,0)=\zeta^0(y)&\mbox{ in }\Omega,
\end{cases}
\end{equation}
and the initial flow map $\zeta^0:=\zeta^0(y):\Omega\to \Omega$ is diffeomorphism that satisfies
\begin{align}  \label{zeta0inta}
\partial\Omega=\zeta^0(\partial\Omega)\;\mbox{ and }\;
\det\nabla \zeta^0=1\;\mbox{ for any }y\in \overline{\Omega}.
\end{align}
Here and in what follows, ``$\det$'' denotes the determinant of a matrix.
We denote the Eulerian coordinates by $(x,t)$ with $x=\zeta(y,t)$, whereas $(y,t)\in \Omega\times\mathbb{R}^{+}$ stands for the Lagrangian coordinates.
In order to switch back and forth from Lagrangian to Eulerian coordinates, we assume that
$\zeta(\cdot,t)$ is invertible and $\Omega=\zeta(\Omega, t)$, which can be achieved
when the flow map $\zeta$ is a small perturbation around the identity map $Id$.
In addition, since $\tilde{u}$ is divergence-free,  the flow map $\zeta$ also satisfies the volume-preserving condition (see \cite{MAJBAL}):
\begin{align}\label{202110121654}
\mathrm{det}\nabla\zeta=1
\end{align}
as well as $\det\nabla \zeta^0=1$.

Define the Lagrangian unknowns by
\begin{equation*}
(u, B, w, p)(y,t)=(\tilde{u}, \tilde{B},\tilde{w}, \tilde{p}+ \lambda|\tilde{B}|^2/2)(\zeta(y,t),t),\quad (y,t)\in \Omega\times\mathbb{R}^+.
\end{equation*}
Then the initial-boundary value problem \eqref{0101}--\eqref{0101c} are reformulated as follows:
\begin{equation}\label{202109221247}
\begin{cases}
\zeta_t=u&\mbox{ in } \Omega ,\\[0.5mm]
 u_t +\nabla_{\mathcal{A}}p- (\mu + \chi)\Delta_{\mathcal{A}}u =  \lambda B\cdot \nabla_{\mathcal{A}} B +2\chi\nabla_{\mathcal{A}}\times w
&\mbox{ in } \Omega ,\\[0.5mm]
B_t-B\cdot\nabla_{\mathcal{A}}u=0
&\mbox{ in } \Omega ,\\[0.5mm]
w_t +4\chi w=2\chi\nabla_{\mathcal{A}}\times u&\mbox{ in } \Omega ,\\[0.5mm]
\mathrm{div}_{\mathcal{A}}u=\mathrm{div}_{\mathcal{A}}B=0 &\mbox{ in } \Omega ,\\[1mm]
(\zeta, u, B, w)|_{t=0}=(\zeta^0, u^0, B^0, w^0):=(\zeta^0, \tilde{u}^0(\zeta^0), \tilde{B}^0(\zeta^0),\tilde{w}^0(\zeta^0))\quad &\mbox{ in } \Omega ,\\[0.5mm]
(\zeta, u)=(y,0) & \mbox{ on }\partial\Omega.
\end{cases}
\end{equation}
Here and in what follows we have written the differential operators $\nabla_{\mathcal{A}}$, $\mathrm{div}_\mathcal{A}$ with their actions given by
$(\nabla_{\mathcal{A}}f)_i:=\mathcal{A}_{ij}\partial_jf$, $\mathrm{div}_{\mathcal{A}}X:=\mathcal{A}_{ij}\partial_j X_{i}$
for appropriate scalar function $f$ and vector function $X:=(X_1,X_2,X_3)^{\top}$, and $\Delta_{\mathcal{A}}f:=\mathrm{div}_{\mathcal{A}}
\nabla_{\mathcal{A}}f$.
It should be noted that the Einstein convention of summation over repeated indices has been used here,
and $\partial_{j}$ denotes the partial derivative with respect to the $j$-th component of the variable $y$, i.e., $\partial_{j}:=\partial_{y_j}$.
The differential operator $\nabla_{\mathcal{A}}\times$ with it action given by $(\nabla_{\mathcal{A}}\times X)_{i}:=\epsilon_{ijk}\mathcal{A}_{jl}\partial_{l}X_{k}$,
where $\epsilon_{ijk}$ denotes the Levi-Civita symbol, i.e.,
\begin{align*}
\epsilon_{ijk}=
\begin{cases}
1&\mbox{ if }\;ijk=123,231,321;\\[0.5mm]
-1&\mbox{ if }\;ijk=312,213,132;\\[0.5mm]
0&\mbox{ others}.
\end{cases}
\end{align*}

We turn to study the equivalently initial-boundary value problem \eqref{202109221247} in Lagrangian coordinates,
and  aim to prove the global well-posedness of the initial-boundary value problem \eqref{202109221247}
around the equilibrium state $(u, B, w)=(0, \bar{B},0)$, where $\bar{B}=(\bar{B}_1,\bar{B}_2,\bar{B}_3)^{\top}$ is a non-zero constant vector with $\bar{B}_3\neq0$.
Specially, from  the differential version of magnetic flux conservation \cite{JFJSOMITIN}, we can see that
\begin{align}\label{01061800}
B=\bar{B}\cdot \nabla \zeta
\end{align}
provided the initial data $(\zeta^0,B^0)$ satisfies the frozen condition
\begin{align}&\label{abjlj0i}
B^0= \bar{B}\cdot \nabla \zeta^0.
\end{align}
It is worth noting that $B$ given by \eqref{01061800} automatically satisfies \eqref{202109221247}$_3$ and $\mathrm{div}_{\mathcal{A}}B=0$.
Based on \eqref{01061800}, one can further calculate that
\begin{align}\label{0131}
B\cdot \nabla_{\mathcal{A}} B=(\bar{B}\cdot \nabla)^2\zeta.
\end{align}

From now on, we define
\begin{equation*}  
\eta:=\zeta-y,\mbox{ i.e., }\zeta=\eta + y,
\end{equation*}
then $\eta$ physically represents the displacement function of particle (labelled by $y$).
Under the assumption \eqref{abjlj0i}, the problem \eqref{202109221247} can be then transformed into
\begin{equation}\label{202609221247}
\begin{cases}
\eta_t=u&\mbox{ in } \Omega ,\\[0.5mm]
u_t +\nabla_{\mathcal{A}}p- (\mu + \chi)\Delta_{\mathcal{A}}u =  \lambda(\bar{B}\cdot \nabla)^2\eta +2\chi\nabla_{\mathcal{A}}\times w&\mbox{ in } \Omega ,\\[0.5mm]
w_t +4\chi w=2\chi\nabla_{\mathcal{A}}\times u&\mbox{ in } \Omega ,\\[0.5mm]
\mathrm{div}_{\mathcal{A}}u=0 &\mbox{ in } \Omega ,\\[0.5mm]
(\eta,u,  w)|_{t=0}=(\eta^0, u^0, w^0) &\mbox{ in } \Omega ,\\
(\eta, u)=(0,0) & \mbox{ on }\partial\Omega
\end{cases}
\end{equation}
and
$$B=\bar{B}\cdot \nabla(\eta+y)\;\;\mbox{and}\;\;\mathcal{A}=(\nabla \eta+\mathbb{I})^{-\top},$$
where $\mathbb{I}$ denotes the $3\times 3$ identity matrix.
Finally, we present some useful properties involving $\mathcal{A}$ which we use repeatedly in what follows.
In view of the definition of $\mathcal{A}$ and \eqref{202110121654}, we can deduce the following relation:
\begin{align}
\label{Akl=0}
\partial_j \mathcal{A}_{ij}=0.
\end{align}
In particular, the relation \eqref{Akl=0} yields
\begin{align}
&\label{diverelation2}
\mathrm{div}_{\partial_t^k\mathcal{A}}\partial_t^lu=
\mathrm{div}(\partial_t^k\mathcal{A}^{\top}\partial_t^lu)
\quad\mbox{for}\;k, l\geqslant0.
\end{align}
The rest of this paper is devoted to providing the global well-posedness of solutions for the problem \eqref{202609221247} with small initial data.

\subsection{Perturbed form}\label{202610101708}
Note that the differential operators in \eqref{202609221247} depend on the entries of $\mathcal{A}$ and their derivatives,
we recast the problem \eqref{202609221247} as a perturbation of the linearized system.
This reformulation fixes the coefficients, thereby facilitating both the energy estimates
and the application of elliptic regularity theory in subsequent sections.
Specifically, we shall use the following linear perturbed formulation:
\begin{equation}\label{202609221247n}
\begin{cases}
\eta_t=u&\mbox{ in } \Omega ,\\
u_t +\nabla p- (\mu + \chi)\Delta u =  \lambda (\bar{B}\cdot \nabla)^2\eta +2\chi\nabla\times w+\mathcal{N}^1&\mbox{ in } \Omega ,\\[0.5mm]
w_t +4\chi w=2\chi\nabla\times u+\mathcal{N}^2&\mbox{ in } \Omega ,\\[0.5mm]
\mathrm{div}u=\mathcal{N}^3 &\mbox{ in } \Omega ,\\[0.5mm]
(\eta,u,  w)|_{t=0}=(\eta^0, u^0, w^0) &\mbox{ in } \Omega ,\\
(\eta, u)=(0,0) & \mbox{ on }\partial\Omega,
\end{cases}
\end{equation}
where we have defined that $\tilde{\mathcal{A}}:={\mathcal{A}}-\mathbb{I}$, and
\begin{align*}
&\mathcal{N}^1:=(\mu + \chi)\big(\mathrm{div}_{\mathcal{A}}\nabla_{\tilde{\mathcal{A}}}u^{\top}+\mathrm{div}_{\tilde{\mathcal{A}}}\nabla u^{\top}\big)-\nabla_{\tilde{\mathcal{A}}} p+2\chi\nabla_{\tilde{\mathcal{A}}}\times w,\\[1mm]
&\mathcal{N}^2:=2\chi\nabla_{\tilde{\mathcal{A}}}\times u,\quad \mathcal{N}^3:=-\mathrm{div}_{\tilde{\mathcal{A}}}u,
\end{align*}
and the operator $\nabla_{\tilde{\mathcal{A}}}$ is defined analogously to $\nabla_{\mathcal{A}}$ with $\mathcal{A}$ replaced by $\tilde{\mathcal{A}}$.
The operator $\mathrm{div}_{\tilde{\mathcal{A}}}$ is defined in a similar manner.

\vspace{3mm}
The rest of this paper is organized as follows.
In Section \ref{Main result}, we first introduce some simplified notations throughout this paper, and then introduce our main result, that is Theorem \ref{thm1}.
Section \ref{preliminaries} is devoted to deriving some preliminary estimates that will be used in the proof of Theorem \ref{thm1}.
In Section \ref{energy estimates}, we derive the energy estimates, including the tangential and normal derivatives energy estimates.
Based on these energy estimates, we show the a \emph{priori} estimate by adapting the  multi-layer energy method in Section \ref{2025thm1}, and finally Theorem \ref{thm1} follows.

\section{Main result}\label{Main result}

Before stating our stability result, we shall introduce some simplified notations throughout this paper.

(1) Basic notations: 
$\overline{\Omega}:=\mathbb{R}^2\times[0,1]$;
$\int\cdot\mathrm{d}y:=\int_{\Omega}\cdot\mathrm{d}y$ denotes the integral over domain $\Omega$;
$a\lesssim b$ means that $a\leqslant cb$ for some ``universal'' constant $c>0$, where the constant $c$ may depend on
some physical parameters such as
$\mu$, $\chi$ and $\lambda$, and may be  different from line to line.
Moreover,
$f_{\mathrm{h}}:=(f_1, f_2)$
for $f=(f_{\mathrm{h}}, f_3)^\top$;
$\mathrm{div}_{\mathrm{h}}f_{\mathrm{h}}:=\partial_1f_1+\partial_2f_2$, and
$\mathrm{d}y_{\mathrm{h}}:=\mathrm{d}y_{1}\mathrm{d}y_{2}$.
$\partial_{\mathrm{h}}^{\alpha}:=\partial_1^{\alpha_1}\partial_2^{\alpha_2}$ for some multiindex of order $\alpha:=(\alpha_1, \alpha_2)$ with $|\alpha|=\alpha_1+\alpha_2$;
$\partial_{\mathrm{h}}^{i}$ denotes $\partial_{\mathrm{h}}^{\alpha}$ for any $\alpha$ satisfying $|\alpha|=i$.
Letters $c_{j}$ ($0\leqslant j\leqslant 10$) are fixed constants which may depend on the parameters.

(2) Simplified notations of function spaces:
\begin{equation*}
\begin{aligned}
&L^p:=L^p (\Omega)=W^{0,p}(\Omega),
\;\;W^{i,2}:=W^{i,2}(\Omega),\;\;H^i:=W^{i,2},\\[1mm]
&H_0^1:=\{f\in H^1~|~f|_{\partial\Omega}=0\;\mbox{in the sense of trace}\},\;\;
H_0^j:=H_0^1\cap H^j,
\end{aligned}
\end{equation*}
where $1<p\leqslant\infty$, and $i\geqslant0$, $j\geqslant1$ are integers.

(3) Simplified Sobolev norms and semi-norms:
\begin{align*}
&\|\cdot\|_{i}:=\|\cdot\|_{W^{i,2}},\;\;
\|\cdot\|_{i,k}:=\sum_{|\alpha|=i} \|\partial_{\mathrm{h}}^{\alpha}\cdot\|_{k},\;\;
\|\cdot\|_{\underline{i},k}:=\sum_{j=0}^i\|\cdot\|_{j,k},
\end{align*}
where $i$, $j$ and $k$ are non-negative integers. For convenience, we denote $\sqrt{\sum_{1\leqslant k\leqslant n}\|f_k\|_{\mathcal{X}}^2}$ by $\|(f_1,\ldots,f_n)\|_{\mathcal{X}}$, where $\|\cdot\|_{\mathcal{X}}$ represents a norm or a semi-norm, and each of $f_1,\ldots,f_n$ may be a scalar-, or vector-valued function.

(4) Lower-order and higher-order energy functionals, denoted by $\mathcal{E}_{L}$ and $\mathcal{E}_{H}$ respectively:
\begin{align*}
&\mathcal{E}_{L}:=\|(\nabla\eta,u)\|_{\underline{3},0}^2
+\|\eta\|_{3}^2+\|(u,w)\|_{2}^2+\|\nabla p\|_{0}^2+\|u_t\|_{0}^2,\\
& \mathcal{E}_{H} :=\|(\nabla\eta,u)\|_{\underline{5},0}^2+\|((\bar{B}\cdot\nabla)\eta, \eta)\|_{5}^2
+\sum_{j=0}^{2}\|\partial_t^ju\|_{5-2j}^2+\sum_{j=0}^{1}\|\nabla\partial_t^jp\|_{3-2j}^2\qquad
\\&\quad\quad\quad+\|w\|_{5}^2+\sum_{j=1}^{2}\|\partial_t^jw\|_{6-2j}^2,
\end{align*}
and the corresponding dissipative functionals $\mathcal{D}_{L}$ and $\mathcal{D}_{H}$ are defined by:
\begin{align*}
&\mathcal{D}_{L} :=\|((\bar{B}\cdot\nabla)\eta, \nabla u)\|_{\underline{3},0}^2
+\|(\eta,u)\|_{3}^2+\|w\|_{2}^2+\|\nabla p\|_{1}^2+\|u_t\|_{1}^2,\\
&\mathcal{D}_{H} :=
\|((\bar{B}\cdot\nabla)\eta,\nabla u)\|^2_{\underline{5},0}+\|((\bar{B}\cdot\nabla)\eta, \eta)\|_{5}^2
+\sum_{j=0}^{3}\|\partial_t^ju\|_{6-2j}^2+\sum_{j=0}^{2}\|\nabla\partial_t^j p\|_{4-2j}^2\\&\quad\quad\quad
+\|w\|_{5}^2+\sum_{j=1}^{2}\|\partial_t^jw\|_{7-2j}^2.
\end{align*}

(5) The total energy functional $\mathcal{G}(t):= \sum_{i=1}^4\mathcal{G}_i(t)$ with
\begin{equation*}
\begin{aligned}
 &\mathcal{G}_1(t):=\sup_{0\leqslant \tau< t}\|\eta(\tau)\|_6^2,\quad
 \mathcal{G}_2(t):=\int_0^t\frac{\|\eta(\tau)\|_6^2}{(1+\tau)^{3/2}} \mathrm{d}\tau,\\[1mm]
& \mathcal{G}_3(t):=\sup_{0\leqslant \tau< t}\mathcal{E}_H(\tau)+\int_0^t
\mathcal{D}_H(\tau)\mathrm{d}\tau,\\[1mm]
&\mathcal{G}_4(t):=\sup_{0\leqslant \tau< t}(1+\tau)^{2}\mathcal{E}_L(\tau)+\int_0^t{(1+\tau)^{3/2}}{\mathcal{D}_{L}}(\tau)\mathrm{d}\tau.\qquad\qquad\qquad\qquad\qquad
\end{aligned} \end{equation*}

\vspace{3mm}
Our global well-posedness result for the problem \eqref{202609221247} is given as follows.

\begin{thm}\label{thm1}
Let the initial data $(\eta^0, u^0, w^0)\in H_0^6\times H_0^5\times H^5$.
There exists a sufficiently small constant $\delta >0$ such that, if the initial data $(\eta^0, u^0, w^0)$ satisfies
\begin{itemize}[\quad (1)]
\item[(1)]$\mathrm{det}(\nabla\eta^0+\mathbb{I})=1$\mbox{ and }$\mathrm{div}_{\mathcal{A}^0}u^0=0$\;\;in $\Omega$;
   \item[(2)]  the compatibility conditions $\partial_t^{j} u|_{t=0}=0$ (j=1,\;2) on $\partial\Omega$;
   \item[(3)] $\sqrt{\| \eta^0 \|_6^2+ \|(u^0,w^0)\|_{5}^2}\leqslant\delta $.
\end{itemize}
Then the initial-boundary value problem \eqref{202609221247}
admits a unique global solution $(\eta,u,w)$ on $[0,\infty)$ with an associated  pressure $p$ (up to a constant).
Moreover, the solution enjoys
\begin{align}\label{1.19}
 \mathcal{G}(\infty):= \sum_{i=1}^4\mathcal{G}_i(\infty)
 \lesssim  \| \eta^0 \|_6^2+\|(u^0,w^0)\|_{5}^2.
\end{align}
Here, $\mathcal{A}^0$ denotes the initial data of $\mathcal{A}$, defined by $\eta^0$, while the positive constant $\delta$ depends on other known physical parameters.
\end{thm}

\begin{rem}\label{rem:11301552}
Let $(\eta,u,w)$ be constructed in Theorem \ref{thm1}. If $\delta$ is sufficiently small, then, for each fixed $t>0$,
\begin{align}
& {\zeta}:=\eta+y: \overline{\Omega} \to   \overline{\Omega}\; \mbox{ is a homeomorphism mapping},\nonumber  \\
& {\zeta}: \Omega  \to \Omega\; \mbox{ is a } C^4\mbox{-diffeomorphic mapping}\nonumber.
\end{align}
Thus, by an inverse transformation of Lagrangian coordinates,
we can recover the global well-posedness of the original problem \eqref{0101}--\eqref{0101c} from Theorem \ref{thm1}.
\end{rem}

\begin{rem}\label{rem:1}
Based on the derivation of the time-decay estimate for $\mathcal{E}_{L}(t)$, one observes that higher regularity of solutions implies a faster time-decay rate for lower-order derivatives. Thus, we can also establish an almost exponential time-decay estimate as demonstrated in \cite{TZWYJGw}.
\end{rem}

Now we briefly sketch the proof of Theorem \ref{thm1}.
The local well-posedness result of the incompressible viscous MHD system without magnetic diffusion and the related incompressible viscous fluids has been established in several works; see, for instance \cite{WYJTIKCT,GYTILW1,JFJSJMFMOSERT}.
Moreover, note that Equations \eqref{202609221247}$_3$ for $w$ is an ODE, from which one can recover the solution $w$ in terms of $\nabla_{\mathcal{A}}\times u$.
Hence, following the arguments for the existence result presented in \cite{GYTILW1,JFJSJMFMOSERT}, we can use a standard approximation scheme for the linearized problem and an iterative method to establish a local-in-time existence result for a unique solution $(\eta,u,w)$ with an associated pressure $p$ to the problem \eqref{202609221247}.
Therefore, by a continuity argument, to prove Theorems \ref{thm1}, it suffices to derive the a priori estimate \eqref{202604281950}, as shown in Proposition \ref{pro:0807}.

Our strategy for establishing the \emph{a priori} estimate is the natural energy evolution.
First, we consider the basic energy identities, which come from testing the momentum equations by $u$ and $\eta$, respectively:
\begin{align*}
&\frac{1}{2}\frac{\mathrm{d}}{\mathrm{d}t}\big(\|u\|_0^2+\lambda\|(\bar{B}\cdot\nabla)\eta\|_{0}^2\big)
+(\mu+\chi)\|\nabla u\|_0^2-2\chi\int\nabla \times w\cdot u\mathrm{d}y=\mathcal{M}^1,\\[1mm]
&\frac{1}{2}\frac{\mathrm{d}}{\mathrm{d}t}\|\sqrt{(\mu+\chi)}\nabla\eta\|_{0}^2
+\lambda\|(\bar{B}\cdot\nabla)\eta\|_{0}^2+\int\eta\cdot u_{t}\mathrm{d}y
-2\chi\int\nabla \times w\cdot \eta\mathrm{d}y=\mathcal{M}^2
\end{align*}
and from testing the micro-rotation equations by $w$:
\begin{align*}
&\frac{1}{2}\frac{\mathrm{d}}{\mathrm{d}t}\|w\|_0^2+4\chi\|w\|_0^2-2\chi\int\nabla \times u\cdot w\mathrm{d}y=\mathcal{M}^3,\qquad\qquad\qquad
\end{align*}
where $\mathcal{M}^{i}\;(i=1,2,3)$ represents the integral involving the nonlinear terms.
These basic energy identities demonstrate that, for the magneto-micropolar fluid system without magnetic diffusion and spin viscosity,
the dissipation provides no direct control of the energy due to
the weak dissipation from the magnetic tension $\lambda(\bar{B}\cdot\nabla)^2\eta$, and thereby leads to the presence of some delicate integrals that cannot be controlled by dissipation alone.
It is worth noting that similar difficulties also arise in the study of surface wave problems for viscous fluids,
where the dissipation is also weaker than the energy.
To overcome this difficulty, Guo--Tice introduced the celebrated two-layer energy method to consider
separately the lower-order energy with decay-in-time  and the bounded higher-order energy  \cite{GYTIAE2,GYTIDAP}.
Such a method has proven to be effective in studying the viscous and non-resistive MHD system (i.e., $w=\chi=0$) as well;
please refer to a series of works by Jiang--Jiang \cite{JFJSJMFMOSERT,JFJSOUI} and Wang \cite{TZWYJGw,WYTIVNMI} for more results in this research.
However, different from these works in Jiang--Jiang and Wang, the energy evolutions for the magneto-micropolar fluid system introduce additional terms involving vorticities.
These terms are linear and anti-symmetric, posing significant challenges due to the lack of spin viscosity.
Very recently, by adapting the two-layer energy method of Guo--Tice and the novel div-curl decomposition for micro-rotation velocity,
Feng--Hong--Zhu \cite{FengSCM} successfully established the global existence of classical solutions to the 3D compressible magneto-micropolar fluid equations in the absence of magnetic diffusion and spin viscosity.

In the spirit of the method of Guo--Tice and the div-curl decomposition technique used in Feng et al., we introduce below a new variant of the multi-layer energy method to resolve our problem. Specifically, we establish three energy inequalities, namely the lower-order energy inequality
\begin{align}\label{202109291440}
\frac{\mathrm{d}}{\mathrm{d}t}\tilde{\mathcal{E}}_{L}+{\mathcal{D}}_{L}\leqslant 0,\qquad\quad
\end{align}
and the higher-order and the highest-order energy inequalities
\begin{align}
&\label{202110151621}
\frac{\mathrm{d}}{\mathrm{d}t}\tilde{\mathcal{E}}_{H}+{\mathcal{D}}_{H}\lesssim
\sqrt{\mathcal{D}_{L}}\|\eta\|_{6}^2,\\
&\frac{\mathrm{d}}{\mathrm{d}t} \overline{\|\eta\|}_{6}^2+\|\eta\|_{6}^2\lesssim  \mathcal{E}_H+\mathcal{D}_{H},\nonumber
\end{align}
where the functionals $\tilde{\mathcal{E}}_{L}$ and $\tilde{\mathcal{E}}_{H}$ are equivalent to ${\mathcal{E}}_{L}$ and ${\mathcal{E}}_{H}$, respectively
and $\overline{\|\eta\|}_{6}^2$ is equivalent to $\|\eta\|_{6}^2$.
Since we need to work with higher-order energy functionals, and the basic energy estimate does not apply to the higher-order normal spatial derivatives of $(q,u,w)$
in Sobolev norms due to the presence of physical boundaries, our proof requires several steps.
Roughly speaking, the argument is divided into the tangential estimates and normal estimates for $(q, u, w)$, all of which are established under the \emph{a priori} assumption \eqref{aprpioses}.

Observe that tangential derivatives (including temporal and horizontal ones) naturally preserve the boundary condition, which
allows for integration by parts and the application of the Friedrichs inequality \eqref{friedrich} and the Poincar\'e-type inequality \eqref{poincaretype} in the energy evolution.
Hence, the strategy employed in previous works such as \cite{FengSCM,JFJSJMFMOSERT,TZWYJGw} can also be applied to our tangential energy estimates (see Lemmas \ref{lem:21092401}--\ref{badiseqinM}).
Having established the tangential estimates for $(\eta,u,w)$,
we turn to the estimates of the normal derivatives. This constitutes a subtle and significant part of the overall proof.

Since $p$ is an unknown rather than a known function of density, and thus lacks an ODE structure  present in the compressible case,
the only viable approach to derive estimates for $p$ is to employ the elliptic regularity theory of the Stokes problem.
This approach simultaneously yields estimates for $u$ (also indeed for $\eta$ as well); see Lemmas \ref{lem:11292030es}--\ref{lem:0426}.
However, to close these estimates for $(\eta,u,p)$, we require control over the high-order derivatives of $w$ previously.
Following the strategy in \cite{FengSCM} regarding the $\mathrm{div}$-$\mathrm{curl}$ decomposition for $w$,
we now establish the normal estimates for $(\nabla\times w,\mathrm{div}w)$, which simultaneously yields the dissipation estimates for
$(\bar{B}\cdot \nabla)^2\eta$ and $u$.
We first derive the evolution equation for $\nabla\times w$ by applying the curl operator to the micro-rotation equations.
By combining this with the momentum equations, we then obtain two ODE systems that decouple the third component from the horizontal components of the
equations for $(\bar{B}\cdot \nabla)^2\eta$, $\partial_3^2u$ and $\nabla \times w$ (see \eqref{2026195637es}--\eqref{2026195638es}).
Let $W=(W_{\mathrm{h}},W_3)^\top:=\nabla\times w$. To illustrate the argument, we focus on the horizontal components for $((\bar{B}\cdot \nabla)^2\eta, \partial_3^2u, W)$:
\begin{equation}\label{2026195637n260421m}
\begin{cases}
\lambda (\bar{B}\cdot \nabla)^2\eta_{\mathrm{h}}+(\mu + \chi)\partial_{3}^2u_{\mathrm{h}}+2\chi W_{\mathrm{h}} = \tilde{\mathcal{N}}_{\mathrm{h}}\qquad\qquad\qquad  & \mbox{in } \Omega ,\\[1mm]
\partial_{t}W_{\mathrm{h}}+4\chi W_{\mathrm{h}}+2\chi\partial_{3}^2u_{\mathrm{h}} = \tilde{\mathcal{M}}_{\mathrm{h}}  & \mbox{in } \Omega.
\end{cases}
\end{equation}
where $\tilde{\mathcal{N}}_{\mathrm{h}} \sim \partial_tu_{\mathrm{h}} + \nabla_{\mathrm{h}} p + \nabla\partial_{\mathrm{h}} u + \emph{h.o.t}$, $\tilde{\mathcal{M}}_{\mathrm{h}} \sim \nabla\partial_{\mathrm{h}} u + \emph{h.o.t}$, and $\emph{h.o.t}$ denotes the nonlinear terms involving $(\eta,u,w,p)$ (see \eqref{2026195637es}).
The presence of the term  $\mu \partial_{3}^2u_{\mathrm{h}}$ allows us to establish positive dissipation for $(\partial_{3}^2u_{\mathrm{h}},W_{\mathrm{h}})$ from the coupling between $\partial_{3}^2u_{\mathrm{h}}$ and $W_{\mathrm{h}}$; see the derivation of \eqref{2092650411535} for details.
It is worth noting that aside from the pressure term $\nabla_{\mathrm{h}} p$, the linear terms in $\tilde{\mathcal{N}}_{\mathrm{h}}$ and $\tilde{\mathcal{M}}_{\mathrm{h}}$
involve only tangential or first-order vertical derivatives of $u$.
This leads us to anticipate that the higher-order estimates for $((\bar{B}\cdot \nabla)^2\eta_{\mathrm{h}},\partial_{3}^2u_{\mathrm{h}},W_{\mathrm{h}})$ can also be reduced to their tangential estimates.
However, since we cannot apply the Stokes estimate to derive bounds for $\nabla_{\mathrm{h}}p$ due to the lack of control over $W$ established previously, the pressure term $\nabla_{\mathrm{h}} p$ presents a significant challenge:
$$\int\partial_{\mathrm{h}}^{j+k}\partial_{3}^{i-j-k}\nabla_{\mathrm{h}} p
\cdot(\bar{B}\cdot \nabla)^2\partial_{\mathrm{h}}^{j+k}\partial_{3}^{i-j-k}u_{\mathrm{h}}\mathrm{d}y.$$
This differs markedly from the setting in \cite{FengSCM}, where the pressure term would otherwise support the dissipation mechanism in conjunction with the magnetic tension in Lagrangian coordinates.
To circumvent this difficulty, we analyze three cases based on $0\leqslant k\leqslant i-j$.
For the case $i=0$, this integral can be bounded by $\|\nabla p\|_{0}\|\nabla u\|_{2}^{1/2}\|\nabla u\|_{0}^{1/2}$.
For the case $i\neq0$ and $k<i-j$, we observe that this integral can be bounded by $\|\nabla \partial_{\mathrm{h}}p\|_{j,i-j-1}\|u_{\mathrm{h}}\|_{j,2+i-j}$, which implies
we may reduce the normal estimates into horizontal ones, utilizing the bound on $\|\nabla \partial_{\mathrm{h}}p\|_{j,i-j-1}$ and Young's inequality.
The most delicate case occurs when $k=i-j$:
$$\int\partial_{\mathrm{h}}^i\nabla_{\mathrm{h}}p\cdot(\bar{B}\cdot \nabla)^2\partial_{\mathrm{h}}^iu_{\mathrm{h}}\mathrm{d}y.$$
In this case we cannot expect to transform the normal derivatives into horizontal ones.
Fortunately, given the presence of sufficient horizontal derivatives, we address this integral by integrating by parts and applying the refined trace estimate \eqref{37190928}, which yields the bound $\|\nabla p\|_{i}\|\nabla u_{\mathrm{h}}\|_{i+1}^{1/2}\|\nabla u_{\mathrm{h}}\|_{\underline{i+1},1}^{1/2}$ (see \eqref{20260411na}).
This integral and the integral for the case $i=0$ are ultimately controlled via Young's inequality combined with the tangential energy estimate  and the bound for $\|u\|_{i+2}+\|\nabla p\|_{i}$.
Notably,  obtaining the dissipation estimate for $(\bar{B}\cdot \nabla)^2\eta_{\mathrm{h}}$, analogous to that for $u$, is also required.
Indeed, to control these difficult integrals, it is essential to ensure that $\|u\|_{i+2}+\|\nabla p\|_{i}\lesssim\sqrt{\mathcal{D}_{L}}$ for $i=1$ and $\lesssim\sqrt{\mathcal{D}_{H}}$ for $i=4$.
In view of \eqref{112922471039}, this estimate is indispensable.
Although the integral involving $\nabla_{\mathrm{h}} p$ presents a challenge, we can handle it using the same strategy outlined above.
Additionally, to derive the normal estimates for $((\bar{B}\cdot \nabla)^2\eta_3,\partial_3^2u_3)$, we invoke the divergence-free condition \eqref{202609221247}$_4$
and the volume-preserving condition \eqref{202110121654}.
This allows us to transform the normal estimates for $((\bar{B}\cdot \nabla)^2\eta_3,\partial_3^2u_3)$ into those for
$((\bar{B}\cdot \nabla)^2\eta_{\mathrm{h}},\partial_3^2u_{\mathrm{h}})$,
combined with the tangential estimates (see \eqref{2092650411339}--\eqref{209265041136}),
which also streamlines the estimates for $W_3$.
The detailed estimates can be found in Lemma \ref{uwnormal}.

Furthermore, by applying the divergence operator to the  micro-rotation equations, we derive the evolution equation for $\mathrm{div} w$,
This equation exhibits a structure analogous to the dissipative ODE $\partial_tf+cf=g$, modulo error terms.
Hence, the energy-dissipation estimate for $\mathrm{div} w$ follows readily; see Lemma \ref{divwnormal} for further details.
Consequently, by collecting the tangential and normal estimates of $(\eta,u,w,p)$ recursively, we obtain the lower-order and higher-order energy inequalities
\eqref{202109291440}--\eqref{202110151621}.

In the next step,  we aim to identify a suitable integrable time-decay for the lower-order dissipation ${\mathcal{D}}_{L}$ so as to close the higher-order energy inequality. Instead of proving the inequality $C\tilde{\mathcal{E}}_{L}\leqslant{\mathcal{D}}_{L}$ directly, we are fortunate to be able to prove the stronger inequality $C\tilde{\mathcal{E}}_{L}^{3/2}\leqslant{\mathcal{D}}_{L}$ by utilizing interpolation between lower-order derivatives and bounded higher-order derivatives. Based on this result, we then derive a differential inequality
$$\frac{\mathrm{d}}{\mathrm{d}t}\tilde{\mathcal{E}}_{L}+C\tilde{\mathcal{E}}_{L}^{3/2}\leqslant 0,$$
which implies
\begin{align}
\label{202109291442}
\tilde{\mathcal{E}}_{L}(t)\lesssim \mathcal{E}_{L}(0)(1+t)^{-2}.
\end{align}
Combining \eqref{202109291440} with \eqref{202109291442}
leads to an integrable time-decay of the dissipation ${\mathcal{D}}_{L}$
in the following sense:
$$\sup_{0\leqslant \tau< t}(1+\tau)^{3/2}\mathcal{E}_{L}(\tau)+\int_0^t{(1+\tau)^{3/2}}{\mathcal{D}_{L}}(\tau)\mathrm{d}\tau
\lesssim{\mathcal{E}}_{L}(0).$$
Full details will be presented in Section \ref{2025thm1}.
Consequently, the use of the multi-layer energy method allows us to derive the \emph{a priori} estimate \eqref{202604281950}, which in conjunction with the local well-posedness result and a standard continuity method, establishes Theorem  \ref{thm1}.

\begin{rem}
Upon completion of the proof, we observe that $\|u\|_{6}+\|\nabla p\|_{4}\lesssim\sqrt{\mathcal{D}_{H}}$.
This means that the highest-order dissipation estimate for $(u,p)$ is controlled by $\mathcal{D}_{H}$,
ensuring that its time integral is uniformly bounded.
In addition, by identifying an integrable time-decay of the lower-order dissipation $\mathcal{D}_{L}$,
we can replace the decay of the lower-order energy in closing the higher-order energy inequality, thereby relaxing the regularity requirements for the initial data.
Hence, our result, to some extent, improves the findings in \cite{JFJSJMFMOSERT,JFJSOUI,TZWYJGw,WYTIVNMI}.
\end{rem}

\section{Preliminaries}\label{preliminaries}
In this subsection, we establish preliminary estimates involving  $\mathcal{A}$, $\mathrm{div}\eta$,
and the nonlinear terms.
By exploiting the embedding inequalities and the product estimates, we derive a series of preliminary estimates for the solution $(\eta,u,w,p)$
under the condition
\begin{equation}\label{201811201630}
\sup_{0\leqslant t<T}\sqrt{\|\eta(t)\|_{6}^2+\|(u,w)(t)\|^2_5}< \delta\in(0,1)
\;\;\mbox{for some}\;\;T>0,
\end{equation}
where  $\delta$ is sufficiently small.
We mention that these estimates under the condition \eqref{201811201630} will be frequently used in the proof of Theorem \ref{thm1}.

\subsection{Estimates involving  $\mathcal{A}$ and $\mathrm{div}\eta$ }
\begin{lem}\label{AestimstfroA}
Under assumption \eqref{201811201630} with sufficiently small $\delta$, we have
\begin{enumerate}
\item[(1)] Estimates for $\mathcal{A}$ and $\tilde{\mathcal{A}}:=\mathcal{A}-\mathbb{I}$: for $0\leqslant i\leqslant 5$, $1\leqslant j\leqslant 3$ and $0\leqslant k\leqslant 7-2j$,
\begin{align}
&\label{aimdse}
\|\mathcal{A}\|_{L^{\infty}} \lesssim 1 ,  \\[0.5mm]
&\label{prtislsafdsfsfds}
\|\tilde{\mathcal{A}}\|_{i} \lesssim\|\eta\|_{i+1},\\
&\label{prtislsafdsfs}
\| \partial_{t}^j\mathcal{A}\|_k\lesssim \sum_{l=0}^{j-1} \|\partial_t^l  u\|_{k+1}.
\end{align}
\item[(2)] Equivalent estimate
 \begin{align}\label{20190614fdsa1957}
\;\;\;\;\|\varphi\|_1 \lesssim \|\nabla_{\mathcal{A}}\varphi\|_0\lesssim \|\varphi\|_1\;\quad\mathrm{for \;any }\; \varphi\in H^1_0.
\end{align}
\end{enumerate}
\end{lem}
\begin{pf}
In view of the definition of $\mathcal{A}$ and \eqref{202110121654}, it is easy to see that
\begin{align}\label{06131422}
\mathcal{A}=(\mathcal{A}_{mn})_{3\times 3}=(A^*_{mn})_{3\times 3},
\end{align} where
$A^{*}_{mn}$ is the algebraic complement minor of $(m,n)$-th entry of matrix $(\partial_m \zeta_n)_{3\times 3}$.
We now define
$$  \tilde{\mathcal{A}}^{{L}}:=\left(\begin{array}{ccc}
\partial_2\eta_2+\partial_3\eta_3 &
-\partial_1\eta_2 &
- \partial_1\eta_3 \\
-\partial_2\eta_1 &
\partial_1\eta_1+\partial_3\eta_3 &
- \partial_2\eta_3 \\
-\partial_3\eta_1 &
- \partial_3\eta_2 &
\partial_1\eta_1+\partial_2\eta_2
        \end{array}\right)
$$
and
$$ \tilde{\mathcal{A}}^{{N}}:=\left(\begin{array}{ccc}
\partial_2\eta_2\partial_3\eta_3-\partial_2\eta_3\partial_3\eta_2 &
\partial_1\eta_3\partial_3\eta_2-\partial_1\eta_2\partial_3\eta_3 &
\partial_1\eta_2\partial_2\eta_3-\partial_1\eta_3\partial_2\eta_2    \\
\partial_2\eta_3\partial_3\eta_1-\partial_2\eta_1\partial_3\eta_3 &
\partial_1\eta_1\partial_3\eta_3-\partial_1\eta_3\partial_3\eta_1 &
\partial_1\eta_3\partial_2\eta_1-\partial_2\eta_3\partial_1\eta_1    \\
\partial_2\eta_1\partial_3\eta_2-\partial_2\eta_2\partial_3\eta_1 &
\partial_1\eta_2\partial_3\eta_1-\partial_1\eta_1\partial_3\eta_2 &
\partial_1\eta_1\partial_2\eta_2-\partial_1\eta_2\partial_2\eta_1
        \end{array}\right).
$$
Consequently, it follows from \eqref{06131422} and $\tilde{\mathcal{A}}:=\mathcal{A}-\mathbb{I}$ that
\begin{align}\label{202112121550}
\tilde{\mathcal{A}}=\tilde{\mathcal{A}}^{{L}}+\tilde{\mathcal{A}}^{{N}}\quad\mbox{and}\quad
\mathcal{A}=\mathbb{I}+\tilde{\mathcal{A}}^{{L}}+\tilde{\mathcal{A}}^{{N}}.
\end{align}
By applying \eqref{embed2} and \eqref{product}, we readily obtain \eqref{aimdse}--\eqref{prtislsafdsfsfds} under \eqref{201811201630} with sufficiently small $\delta$.
Moreover, recalling that $\eta_t=u$,
and employing the product estimate \eqref{product} and the assumption \eqref{201811201630}, we infer that
$$  \begin{aligned}
&   \|\partial_t^l \mathcal{A}_{mn}\|_k \lesssim  \sum_{r=0}^{l-1} \|\partial_t^r  u\|_{k+1}\quad
\mbox{for}\;1\leqslant l\leqslant j\leqslant3,
\end{aligned}$$
which yields \eqref{prtislsafdsfs} immediately.

Finally, in view of \eqref{prtislsafdsfsfds} and the Friedrichs inequality \eqref{friedrich}, we have
$$\begin{aligned}
&\|\nabla_{\mathcal{A}}\varphi\|_{0}\lesssim\|\nabla_{\tilde{\mathcal{A}}} \varphi\|_{0}+\|\nabla\varphi\|_{0}\lesssim\|\varphi\|_1,\\[1mm]
&\|\varphi\|_1\lesssim\|\nabla\varphi\|_{0}\lesssim\|\nabla_{\tilde{\mathcal{A}}} \varphi\|_{0}+\|\nabla_{\mathcal{A}}\varphi\|_{0}\lesssim\|\eta\|_3\|\varphi\|_1+\|\nabla_{\mathcal{A}}\varphi\|_{0},
\end{aligned}$$
which implies \eqref{20190614fdsa1957} for sufficiently small $\delta$.
This completes the proof.
\hfill $\Box$
\end{pf}

The following lemma gives bounds in terms of $\mathrm{div}\eta$.
\begin{lem}\label{201811202046}
Under assumption \eqref{201811201630} with sufficiently small $\delta$, we have
\begin{align}
 &\label{11202054}
 \|\mathrm{div}\eta\|_i \lesssim\|\eta\|_{3}\|\eta\|_{i+1}\quad\mathrm{for}\;\;0\leqslant i\leqslant 5. 
\end{align}
In particular,
\begin{align}
&\label{202209210921nn}
\|(\bar{B}\cdot\nabla)\mathrm{div}\eta\|_{4}\lesssim\|\eta\|_{5}\|(\bar{B}\cdot\nabla)\eta\|_{5}.\quad
\end{align}
\end{lem}
\begin{pf}
In view of \eqref{202110121654}, we expand the determinant $\mathrm{det}(\nabla\eta+\mathbb{I})$ to infer that
\begin{align*}
1=\mathrm{det}\nabla\zeta=\mathrm{det}(\nabla\eta+\mathbb{I})
=1+\mathrm{div}\eta+r_2^{\eta}+r_3^{\eta},
\end{align*}
where
$$\begin{aligned}
&r_2^{\eta}:=-\partial_3\eta_2\partial_2\eta_3-\partial_3\eta_1\partial_1\eta_3-\partial_2\eta_1\partial_1\eta_2
+\partial_2\eta_2\partial_3\eta_3+\partial_1\eta_1\partial_3\eta_3+\partial_1\eta_1\partial_2\eta_2,\\[1mm]
&r_3^{\eta}:=\partial_1\eta_1(\partial_2\eta_2\partial_3\eta_3-\partial_2\eta_3\partial_3\eta_2)
-\partial_2\eta_1(\partial_1\eta_2\partial_3\eta_3-\partial_1\eta_3\partial_3\eta_2)
+\partial_3\eta_1(\partial_1\eta_2\partial_2\eta_3-\partial_1\eta_3\partial_2\eta_2).
\end{aligned}$$
Consequently, it follows that
\begin{align*}
&\mathrm{div}\eta=-(r_2^{\eta}+r_3^{\eta}),\nonumber\\[1mm]
&\mathrm{div}(\bar{B}\cdot\nabla)\eta=-((\bar{B}\cdot\nabla)r_2^{\eta}+(\bar{B}\cdot\nabla)r_3^{\eta})\nonumber.
\end{align*}
Finally, taking the $\|\cdot\|_{i}$ and $\|\cdot\|_{4}$ norms of the above two identities, respectively,
and using \eqref{product}, we arrive at \eqref{11202054} and \eqref{202209210921nn} under the assumption \eqref{201811201630} with sufficiently small $\delta$.
This completes the proof.
\hfill $\Box$
\end{pf}

\subsection{Estimates involving the nonlinear terms}

Next, we establish the estimates of the aforementioned nonlinear terms in \eqref{202609221247n}, which will be utilized in the derivation of energy evolution.
\begin{lem}\label{lem:nonlinear}
Under assumption \eqref{201811201630} with sufficiently small $\delta$, we have
\begin{enumerate}[\quad (1)]
  \item  Estimates for $\mathcal{N}^1 $:
\begin{align}\label{06011711jumpv}
&\|\mathcal{N}^1 \|_j
\lesssim
\begin{cases}
\| \eta \|_{3}(\|u\|_{j+2}+\|(\nabla\times w,\nabla p)\|_{j})&\mathrm{for }\;\; 0\leqslant j\leqslant 1;\\[1mm]
\| \eta \|_{3}(\|u\|_{j+2}+\|(\nabla\times w,\nabla p)\|_{j})\\
\;+(\|u\|_3+\|(\nabla\times w,\nabla p)\|_{1})\|\eta\|_{j+2}
&\mathrm{for }\;\; 2\leqslant j\leqslant 4.
\end{cases}\qquad\quad
\end{align}
\item  Estimates for $\mathcal{N}^2 $:
\begin{align}
&\label{11210836}
\|(\mathcal{N}^2 ,\mathcal{N}^3) \|_{j}\lesssim
\begin{cases}
\| \eta \|_{3}\|u\|_{j+1}
\quad&\mathrm{ for }\;\; 0\leqslant j\leqslant 2;\\[0.5mm]
\|\eta\|_3\|u\|_{j+1}+\|u\|_3\|\eta\|_{j+1}
\quad&\mathrm{ for }\;\; 3\leqslant j\leqslant 5.
\end{cases}\qquad\qquad
\end{align}
\item Estimates for the temporal derivatives of $ \mathcal{N}^1 $, $ \mathcal{N}^2 $ and $ \mathcal{N}^3 $:
\begin{align}
&\label{2021081110937}
\|\partial_t\mathcal{N}^1 \|_1\lesssim\|(\eta,u)\|_3\sum_{l=0}^1\left(\|\partial_t^lu\|_{3}+\|\partial_t^l(\nabla p,\nabla w)\|_{1}\right),\\
&\label{202108110942nm}
\|\partial_t^j\mathcal{N}^1\|_{4-2j}\lesssim\|(\eta,u)\|_4\sum_{l=0}^j\left(\|\partial_t^lu\|_{6-2j}+\|\nabla\partial_t^lp\|_{4-2j}\right),\\
& \label{10202100m}\|\partial_t^{j}(\mathcal{N}^2 ,\mathcal{N}^3)\|_k\lesssim
 \|u\|_3\sum_{  l=0}^{j-1}\|\partial_t^l u\|_{k+1}
 +\|\eta\|_{4}(\|\partial_t^j u\|_{k+1} +\|\partial_t^j u\|_{3}),\qquad
\end{align}
where $1\leqslant j\leqslant 2$ and $0\leqslant k\leqslant 5-2j$.
\end{enumerate}
\end{lem}
\begin{pf}
Recalling the definitions of $\mathcal{N}^1$, $\mathcal{N}^2$ and $\mathcal{N}^3$, we identify their principal terms as
\begin{align*}
&\mathcal{N}^1\sim\mathcal{A}\nabla(\tilde{\mathcal{A}}\nabla u)+\tilde{\mathcal{A}}\nabla^2 u+\tilde{\mathcal{A}}\nabla p+\tilde{\mathcal{A}}\nabla w,
\quad\mathcal{N}^2\sim\tilde{\mathcal{A}}\nabla u\quad\mbox{and}
\quad\mathcal{N}^3\sim\tilde{\mathcal{A}}\nabla u.
\end{align*}
By virtue of the preliminary estimates \eqref{aimdse}--\eqref{prtislsafdsfs}, the product estimate \eqref{product} and the interpolation inequality,
we readily obtain \eqref{06011711jumpv}--\eqref{11210836}.
Furthermore, recalling that $\eta_t=u$, we can similarly deduce \eqref{2021081110937}--\eqref{10202100m}.
This completes the proof.
\hfill $\Box$
\end{pf}

In order to exploit the nonlinear structure of \eqref{202609221247} for deriving the tangential energy  estimates
of the temporal derivatives as it is unable to control the interaction between $\partial_t^2p$ and $\mathrm{div}\partial_t^3u$,
we apply $\partial_t^{j}$ ($1\leqslant j\leqslant3$) to \eqref{202609221247} to obtain
\begin{equation}\label{202609221247temporal}
\begin{cases}
\partial_t^{j+1}u +\nabla_{\mathcal{A}}\partial_t^{j}p- (\mu + \chi)\Delta_{\mathcal{A}}\partial_t^{j}u = \lambda (\bar{B}\cdot\nabla)^2\partial_t^{j-1}u +2\chi\nabla_{\mathcal{A}}\times \partial_t^{j}w +\mathcal{N}^{t,j} &\mbox{ in } \Omega ,\\[0.5mm]
\partial_t^{j+1}w +4\chi \partial_t^{j}w=2\chi\nabla_{\mathcal{A}}\times \partial_t^{j}u +\mathcal{M}^{t,j} &\mbox{ in } \Omega ,\\[0.5mm]
\mathrm{div}_{\mathcal{A}}\partial_t^{j}u=\mathrm{div}D^{t,j}_u &\mbox{ in } \Omega ,\\[0.5mm]
(\partial_t^{j-1}u, \partial_t^{j}u)=(0,0) & \mbox{ on }\partial\Omega,
\end{cases}
\end{equation}
where we have defined that
\begin{align}
&\mathcal{N} ^{t,j}:= \sum_{l=1}^jC_{j}^l\big(
(\mu+\chi)\mathrm{div}_{{\mathcal{A}}}\nabla_{\partial_t^l {\mathcal{A}}}(\partial_t^{j-l}u)^{\top}
+(\mu+\chi)\mathrm{div}_{\partial_t^l {\mathcal{A}}}\partial_t^{j-l}(\nabla_{\mathcal{A}}u^{\top})\qquad\nonumber\\
&\qquad\qquad\qquad\;+2\chi\nabla_{\partial_t^l\mathcal{A}}\times \partial_t^{j-l}w
-\nabla_{\partial_t^l\mathcal{A}}\partial_t^{j-l}p\big),\nonumber\\
&\mathcal{M} ^{t,j} :=2\chi\sum_{l=1}^jC_j^l\nabla_{\partial_t^l\mathcal{A}}\times \partial_t^{j-l}u,\nonumber\\
&\label{partialusd} D^{t,j}_u:= -\sum_{l=1}^{j}C_j^{l}\partial_t^{l}{\mathcal{A}}^{\top}\partial_t^{j-l} u.
\end{align}
Here, the notation $C_j^{l}$ denotes the number of $l$-combinations from a given set of $j$ elements, and we have also used the relation \eqref{diverelation2} in \eqref{partialusd}.
We now establish the estimates for the nonlinear terms $\mathcal{N} ^{t,j}$, $ \mathcal{M} ^{t,j}$ and $D^{t,j}_u$ as follows.
\begin{lem}
\label{lem:0933}
 Under assumption \eqref{201811201630} with sufficiently small $\delta$, there holds that
\begin{align}
&\label{202108082010}
\|\mathcal{N}^{t,1}  \|_i + \|\mathcal{M} ^{t,1} \|_{i}+\| D^{t,1}_u\|_i\lesssim \|u\|_{3}(\|u\|_{2+i}+\|(\nabla p,\nabla w)\|_{i})\quad\mathrm{for}\;\;0\leqslant i\leqslant2,\\[1.5mm]
&\label{badiseqin3nes1n0}
\|\mathcal{N}^{t,2} \|_0+\|\mathcal{M}^{t,2} \|_0+\|D^{t,2}_u\|_0+\|D^{t,3}_u\|_0
\nonumber\\
&\lesssim
\big(\|(\eta,u)\|_3+\|(\nabla w,\nabla p)\|_{2}+\|u_t\|_{1}\big)\left(\|u_t\|_{2}+\|u_{tt}\|_1+\|(\nabla p_t,\nabla w_t)\|_{0}\right).
\end{align}
\end{lem}
\begin{pf}
Similar to the derivation of Lemma \ref{lem:nonlinear},
we easily deduce the estimates \eqref{202108082010}--\eqref{badiseqin3nes1n0} via the product estimates \eqref{product} and the preliminary estimates \eqref{aimdse}--\eqref{prtislsafdsfs}.
\hfill $\Box$
\end{pf}

\section{A priori energy estimates}\label{energy estimates}
This section is devoted to deriving the \emph{a priori} energy estimates for the solution, which is  the key step in the proof of Theorem \ref{thm1}.
To this end, we let $(\eta,u,w,p)$ be the solution
of the problem \eqref{202609221247} such that
\begin{equation}\label{aprpioses}
\sqrt{\mathcal{G}_1(T)+ \sup_{t\in[0, T]}\mathcal{E}_H(t)}\leqslant \delta
\end{equation}
holds for some $T>0$ and sufficiently small $\delta\in(0,1)$.
Furthermore, we assume that the solution possesses proper regularity, so that the procedure of formal calculations makes sense.
Clearly, the preliminary estimates established in Section \ref{preliminaries} remain valid under  assumption \eqref{aprpioses}.

\subsection{Energy estimates for the horizontal derivatives}
In this subsection, we derive the energy estimates for the horizontal derivatives of $(\eta,u,w)$.
\begin{lem}\label{lem:21092401}
Under assumption \eqref{aprpioses} with sufficiently small $\delta$, there holds that
\begin{align}
&\label{202109231702}
\frac{\mathrm{d}}{\mathrm{d}t}\big(\|(\partial_\mathrm{h}^{i} u,\partial_\mathrm{h}^{i} w)\|^2_0+\lambda\|(\bar{B}\cdot\nabla)\partial_\mathrm{h}^{i}\eta\|_{0}^2\big)
+c\|(\nabla \partial_\mathrm{h}^{i} u,\partial_\mathrm{h}^{i}w)\|_0^2\nonumber\\[1mm]
&\lesssim\begin{cases}
\sqrt{\mathcal{E}_{H}}{\mathcal{D}}_{L}
& \;\;\;\mathrm{ for }\;\; 0\leqslant i\leqslant3;\\[1mm]
\sqrt{\mathcal{E}_{H}}{\mathcal{D}}_{H}+\|(\eta,u)\|_{3}\|\eta\|_{6}^2\quad
& \;\;\;\mathrm{ for }\;\; 4\leqslant i\leqslant5.
\end{cases}
\end{align}
\end{lem}
\begin{pf}
Applying $\partial_{\mathrm{h}}^{i}$ for $0\leqslant i\leqslant5$ to \eqref{202609221247n}, we have
\begin{equation}\label{202609221247nh}
\begin{cases}
\partial_\mathrm{h}^{i}\eta_t=\partial_\mathrm{h}^{i}u&\mbox{ in } \Omega ,\\[0.5mm]
\partial_\mathrm{h}^{i}u_t +\nabla \partial_\mathrm{h}^{i}p- (\mu + \chi)\Delta \partial_\mathrm{h}^{i}u =  \lambda (\bar{B}\cdot \nabla)^2\partial_\mathrm{h}^{i}\eta +2\chi\nabla\times \partial_\mathrm{h}^{i}w+\partial_\mathrm{h}^{i}\mathcal{N}^1&\mbox{ in } \Omega ,\\[0.5mm]
\partial_\mathrm{h}^{i}w_t +4\chi \partial_\mathrm{h}^{i}w=2\chi\nabla\times \partial_\mathrm{h}^{i}u+\partial_\mathrm{h}^{i}\mathcal{N}^2&\mbox{ in } \Omega ,\\[0.5mm]
\mathrm{div}\partial_\mathrm{h}^{i}u=\partial_\mathrm{h}^{i}\mathcal{N}^3 &\mbox{ in } \Omega ,\\[0.5mm]
(\partial_\mathrm{h}^{i}\eta, \partial_\mathrm{h}^{i}u)=(0,0) & \mbox{ on }\partial\Omega.
\end{cases}
\end{equation}
Taking the inner product of \eqref{202609221247nh}$_2$ with $\partial_{\mathrm{h}}^{i}u$ in $L^2$,
integrating by parts over $\Omega$, and then using \eqref{202609221247nh}$_1$ and the boundary condition \eqref{202609221247nh}$_5$, we arrive at
\begin{align}
&\label{202312271537}
\frac{1}{2}\frac{\mathrm{d}}{\mathrm{d}t}\big(\|\partial_\mathrm{h}^{i}u\|_0^2+\lambda\|(\bar{B}\cdot\nabla)\partial_\mathrm{h}^{i}\eta\|_{0}^2\big)
+(\mu+\chi)\|\nabla \partial_\mathrm{h}^{i}u\|_0^2\nonumber\\[1mm]
&=\int \partial_\mathrm{h}^{i}p\partial_\mathrm{h}^{i}\mathcal{N}^3\mathrm{d}y
+2\chi\int\nabla \times \partial_\mathrm{h}^{i}w\cdot \partial_\mathrm{h}^{i}u\mathrm{d}y
+\int\partial_\mathrm{h}^{i}\mathcal{N}^{1}\cdot \partial_\mathrm{h}^{i}u \mathrm{d}y.\qquad
\end{align}
In a similar manner, taking the inner product of \eqref{202609221247nh}$_3$ with $\partial_{\mathrm{h}}^{i}w$ in $L^2$, we obtain
\begin{align}\label{2026b1140}
&\frac{1}{2}\frac{\mathrm{d}}{\mathrm{d}t}\|\partial_\mathrm{h}^{i}w\|_0^2+4\chi\|\partial_\mathrm{h}^{i}w\|_0^2
=2\chi\int\nabla \times \partial_\mathrm{h}^{i}u\cdot \partial_\mathrm{h}^{i}w\mathrm{d}y
+\int\partial_\mathrm{h}^{i}\mathcal{N}^{2}\cdot \partial_\mathrm{h}^{i}w \mathrm{d}y.
\end{align}
Adding \eqref{2026b1140} to \eqref{202312271537} then gives
\begin{align}\label{2023122715371706}
&\frac{1}{2}\frac{\mathrm{d}}{\mathrm{d}t}\big(\|(\partial_\mathrm{h}^{i}u,\partial_\mathrm{h}^{i}w)\|_0^2
+\lambda\|(\bar{B}\cdot\nabla)\partial_\mathrm{h}^{i}\eta\|_{0}^2\big)\nonumber\\[1mm]
&\;+(\mu+\chi)\|\nabla \partial_\mathrm{h}^{i}u\|_0^2+4\chi\|\partial_\mathrm{h}^{i}w\|_0^2
-2\chi\int\nabla \times \partial_\mathrm{h}^{i}w\cdot \partial_\mathrm{h}^{i}u\mathrm{d}y
-2\chi\int\nabla \times \partial_\mathrm{h}^{i}u\cdot \partial_\mathrm{h}^{i}w\mathrm{d}y\nonumber\\[1mm]
&=\int\partial_\mathrm{h}^{i}\mathcal{N}^{1}\cdot \partial_\mathrm{h}^{i}u \mathrm{d}y
+\int\partial_\mathrm{h}^{i}\mathcal{N}^{2}\cdot \partial_\mathrm{h}^{i}w \mathrm{d}y
+\int \partial_\mathrm{h}^{i}p\partial_\mathrm{h}^{i}\mathcal{N}^3\mathrm{d}y :=\sum_{k=1}^3I_{k}.
\end{align}

By virtue of the boundary condition $\partial_\mathrm{h}^{i}u\big|_{\partial\Omega}=0$ and the following two vector identities
$$\nabla \times w\cdot u=\nabla \times u\cdot w+\mathrm{div}(w\times u),\quad\Delta u=\nabla\mathrm{div}u-\nabla\times\nabla\times u,$$
one can readily verify that
\begin{align}\label{2026g}
&\int\nabla \times \partial_{\mathrm{h}}^{i}w\cdot \partial_{\mathrm{h}}^{i}u\mathrm{d}y\nonumber\\
&=\int\nabla \times \partial_{\mathrm{h}}^{i}u\cdot \partial_{\mathrm{h}}^{i}w\mathrm{d}y
+\int\mathrm{div}( \partial_{\mathrm{h}}^{i}w\times \partial_{\mathrm{h}}^{i}u)\mathrm{d}y
=\int\nabla \times \partial_{\mathrm{h}}^{i}u\cdot \partial_{\mathrm{h}}^{i}w\mathrm{d}y\qquad\qquad\quad
\end{align}
and
\begin{align}
&\label{2026hn}
\|\nabla \partial_{\mathrm{h}}^{i}u\|_{0}=\|\nabla\times \partial_{\mathrm{h}}^{i}u\|_{0}+\|\mathrm{div}\partial_{\mathrm{h}}^{i}u\|_0^2,\\[1mm]
&\label{2026h}
\|\nabla \partial_{\mathrm{h}}^{i}u\|_0^2 +4\|\partial_{\mathrm{h}}^{i}w\|_0^2
-4\int\nabla \times \partial_{\mathrm{h}}^{i}u\cdot \partial_{\mathrm{h}}^{i}w\mathrm{d}y
=\|\nabla \times \partial_{\mathrm{h}}^{i}u-2 \partial_{\mathrm{h}}^{i}w\|_0^2
+\|\mathrm{div}\partial_{\mathrm{h}}^{i}u\|_0^2.
\end{align}
Additionally, noting that for any $\epsilon>0$, the following inequality holds:
\begin{align*}
&\|\partial_{\mathrm{h}}^{i}w\|_{0}^2\lesssim\|\nabla \times \partial_{\mathrm{h}}^{i}u-2\partial_{\mathrm{h}}^{i}w\|_{0}^2
+\epsilon \|\nabla\times \partial_{\mathrm{h}}^{i}u\|_{0}.
\end{align*}
Invoking the above four estimates, we can deduce from \eqref{2023122715371706} that
there exists a suitably small constant $c_0> 0$ such that
\begin{align}\label{202604251735}
&\frac{1}{2}\frac{\mathrm{d}}{\mathrm{d}t}\big(\|(\partial_\mathrm{h}^{i}u,\partial_\mathrm{h}^{i}w)\|_0^2
+\lambda\|(\bar{B}\cdot\nabla)\partial_\mathrm{h}^{i}\eta\|_{0}^2\big)
+c_0\|(\nabla\partial_{\mathrm{h}}^{i}u, \partial_{\mathrm{h}}^{i}w)\|_{0}^2\lesssim\sum_{k=1}^3I_{k}.
\end{align}

We now proceed to estimate $I_k$ in sequence.
First, using the H\"older's inequality, the estimate \eqref{06011711jumpv} and integrating by parts in $y_{\mathrm{h}}$ when $i\geqslant2$, we obtain
\begin{align}\label{202604251942}
|I_1|&\lesssim
\begin{cases}
\|\mathcal{N}^1\|_{0}\|u\|_{0}
& \hbox{ for }\; i=0;\\[1mm]
\|\mathcal{N}^1\|_{1,0}\|u\|_{2i-1,0}
& \hbox{ for }\; 1\leqslant i\leqslant3;\\[1mm]
\|\mathcal{N}^1\|_{i-1,0}\|u\|_{i+1,0}
& \hbox{ for }\; 4\leqslant i\leqslant5
\end{cases}
\lesssim\begin{cases}
\sqrt{\mathcal{E}_{H}}{\mathcal{D}}_{L}
& \;\hbox{ for }\; 0\leqslant i\leqslant3;\\[1mm]
\sqrt{\mathcal{E}_{H}}{\mathcal{D}}_{H}
& \;\hbox{ for }\; 4\leqslant i\leqslant5.
\end{cases}
\end{align}
Similarly, by \eqref{11210836}, we can have
\begin{align}\label{202604251950}
|I_2|&\lesssim
\begin{cases}
\|\mathcal{N}^2\|_{0}\|w\|_{0}
& \hbox{for }\; i=0;\\[1mm]
\|\mathcal{N}^2\|_{1,0}\|w\|_{2i-1,0}
& \hbox{for }\; 1\leqslant i\leqslant3;\\[1mm]
\|\mathcal{N}^2\|_{i-1,0}\|w\|_{i+1,0}
& \hbox{for }\;i=4;\\[1mm]
\|\mathcal{N}^2\|_{5,0}\|w\|_{5,0}
& \hbox{for }\; i=5
\end{cases}
\lesssim\begin{cases}
\sqrt{\mathcal{E}_{H}}{\mathcal{D}}_{L}
& \hbox{for }\; 0\leqslant i\leqslant3;\\[1mm]
\sqrt{\mathcal{E}_{H}}{\mathcal{D}}_{H}
& \hbox{for }\; i=4;\\[1mm]
\|(\eta,u)\|_{3}\|(\eta,u)\|_{6}\|w\|_{5,0}
& \hbox{for }\; i=5.
\end{cases}
\end{align}
Finally, we proceed to estimate the term $I_3$.
When $i=0$, we employ \eqref{Akl=0} with $\tilde{\mathcal{A}}$ in place of $\mathcal{A}$, \eqref{202609221247n}$_4$, \eqref{prtislsafdsfsfds},
and perform an integration by parts to obtain
\begin{align}\label{202604252050}
I_3
=-\int p\mathrm{div}_{\tilde{\mathcal{A}}}u\mathrm{d}y
=\int\nabla p\cdot(\tilde{\mathcal{A}}^{\top}u)\mathrm{d}y
\lesssim\|\eta\|_3\|u\|_1\|\nabla p\|_0.
\end{align}
In the remaining case where $1\leqslant i\leqslant5$,  an application of \eqref{11210836} yields
\begin{align}\label{202604252056}
|I_3|&\lesssim\|\nabla p\|_{i-1,0}\|\mathcal{N}^3\|_{i,0}\nonumber\\[1.5mm]
&\lesssim\begin{cases}
\sqrt{\mathcal{E}_{H}}{\mathcal{D}}_{L}
& \;\hbox{ for }\; 1\leqslant i\leqslant3;\\[1mm]
\sqrt{\mathcal{E}_{H}}{\mathcal{D}}_{H}+\|(\eta,u)\|_{3}\|(\eta,u)\|_{6}\|\nabla p\|_{4,0}
& \;\hbox{ for }\; 4\leqslant i\leqslant5.
\end{cases}
\end{align}
Consequently, plugging \eqref{202604251942}--\eqref{202604252056} into \eqref{202604251735} and finally using Young's inequality, we arrive at \eqref{202109231702}.
This completes the proof.
\hfill$\Box$
\end{pf}

\begin{lem}\label{lem:26092401}
Under assumption \eqref{aprpioses} with sufficiently small $\delta$, there holds that
\begin{align}
&\label{202109231700}
\frac{1}{2}\frac{\mathrm{d}}{\mathrm{d}t}
\left(\int\big(2\partial_{\mathrm{h}}^{i}u\cdot\partial_{\mathrm{h}}^{i}\eta\mathrm{d}y
+\int\partial_{\mathrm{h}}^{i}w\cdot\partial_{\mathrm{h}}^{i}\nabla\times\eta\mathrm{d}y
+\mu\|\nabla\partial_\mathrm{h}^{i} \eta\|_0^2+\chi\|\mathrm{div}\partial_\mathrm{h}^{i} \eta\|_0^2\right)
+\lambda\|(\bar{B}\cdot\nabla)\partial_{\mathrm{h}}^i\eta\|_{0}^2\nonumber\\[1.5mm]
&\lesssim\|(\nabla\partial_{\mathrm{h}}^{i} u, \partial_{\mathrm{h}}^{i} w)\|_0^2
+\begin{cases}
\sqrt{\mathcal{E}_{H}}{\mathcal{D}}_{L}
& \;\mathrm{ for }\;\; 0\leqslant i\leqslant3;\\[1mm]
\sqrt{\mathcal{E}_{H}}{\mathcal{D}}_{H}+(\|(\eta,u)\|_3+\|(\nabla w,\nabla p)\|_{1})\|\eta\|_{6}^2
& \;\mathrm{ for }\;\; 4\leqslant i\leqslant5.
\end{cases}
\end{align}
\end{lem}
\begin{pf}
Taking the inner product of \eqref{202609221247nh}$_2$ with $\partial_{\mathrm{h}}^{i}\eta$ in $L^2$,
integrating by parts over $\Omega$, and then using \eqref{202609221247nh}$_1$ together with \eqref{202609221247nh}$_5$,
we have
\begin{align}
&\label{202604260805}
\frac{1}{2}\frac{\mathrm{d}}{\mathrm{d}t}
\left(2\int\partial_{\mathrm{h}}^{i}u\cdot\partial_{\mathrm{h}}^{i}\eta\mathrm{d}y
+(\mu+\chi)\|\nabla\partial_\mathrm{h}^{i} \eta\|_0^2\right)
+\lambda\|(\bar{B}\cdot\nabla)\partial_{\mathrm{h}}^i\eta\|_{0}^2-2\chi\int\nabla \times \partial_\mathrm{h}^{i}w\cdot \partial_\mathrm{h}^{i}\eta\mathrm{d}y\nonumber\\[1mm]
&=\|\partial_{\mathrm{h}}^{i} u\|_0^2+\int\partial_\mathrm{h}^{i}\mathcal{N}^{1}\cdot \partial_\mathrm{h}^{i}\eta \mathrm{d}y
+\int \partial_\mathrm{h}^{i}p\mathrm{div} \partial_\mathrm{h}^{i}\eta\mathrm{d}y.
\end{align}
By virtue of \eqref{202609221247nh}$_3$ and \eqref{202609221247nh}$_1$, we rewrite the term $-2\chi\int\nabla \times \partial_\mathrm{h}^{i}w\cdot \partial_\mathrm{h}^{i}\eta\mathrm{d}y$ as
\begin{align*}
&-2\chi\int\nabla \times \partial_\mathrm{h}^{i}w\cdot \partial_\mathrm{h}^{i}\eta\mathrm{d}y
\nonumber\\[1mm]&
=-2\chi\int\partial_\mathrm{h}^{i}w\cdot\nabla \times  \partial_\mathrm{h}^{i}\eta\mathrm{d}y
=\frac{1}{2}\int\partial_\mathrm{h}^{i}\big(w_t-2\chi\nabla \times u-\mathcal{N}^2\big)\cdot\nabla \times  \partial_\mathrm{h}^{i}\eta\mathrm{d}y\nonumber\\[1.5mm]
&=\frac{1}{2}\frac{\mathrm{d}}{\mathrm{d}t}\left(\int \partial_\mathrm{h}^{i}w\cdot\nabla \times  \partial_\mathrm{h}^{i}\eta\mathrm{d}y
-\chi\|\nabla \times  \partial_\mathrm{h}^{i}\eta\|_{0}^2\right)
-\frac{1}{2}\int \left(\partial_\mathrm{h}^{i}w\cdot\nabla \times  \partial_\mathrm{h}^{i}u+\partial_\mathrm{h}^{i}\mathcal{N}^2\cdot\nabla \times  \partial_\mathrm{h}^{i}\eta\right)\mathrm{d}y.
\end{align*}
Similar to \eqref{2026hn}, it also holds that
\begin{align*}
&
\|\nabla \partial_{\mathrm{h}}^{i}\eta\|_{0}=\|\nabla\times \partial_{\mathrm{h}}^{i}\eta\|_{0}+\|\mathrm{div}\partial_{\mathrm{h}}^{i}\eta\|_0^2.
\end{align*}
Inserting the above two identities into \eqref{202604260805} yields
\begin{align}
&\label{202604260821}
\frac{1}{2}\frac{\mathrm{d}}{\mathrm{d}t}
\left(2\int\partial_{\mathrm{h}}^{i}u\cdot\partial_{\mathrm{h}}^{i}\eta\mathrm{d}y
+\int\partial_{\mathrm{h}}^{i}w\cdot\partial_{\mathrm{h}}^{i}\nabla\times\eta\mathrm{d}y
+\mu\|\nabla\partial_\mathrm{h}^{i} \eta\|_0^2+\chi\|\mathrm{div}\partial_\mathrm{h}^{i} \eta\|_0^2\right)
+\lambda\|(\bar{B}\cdot\nabla)\partial_{\mathrm{h}}^i\eta\|_{0}^2\nonumber\\[1mm]
&=\|\partial_{\mathrm{h}}^{i} u\|_0^2
+\frac{1}{2}\int \partial_\mathrm{h}^{i}w\cdot\nabla \times  \partial_\mathrm{h}^{i}u\mathrm{d}y
+\int\partial_\mathrm{h}^{i}\mathcal{N}^{1}\cdot \partial_\mathrm{h}^{i}\eta \mathrm{d}y
+\frac{1}{2}\int\partial_\mathrm{h}^{i}\mathcal{N}^2\cdot\nabla \times  \partial_\mathrm{h}^{i}\eta\mathrm{d}y\nonumber\\[1mm]
&\quad\;
+\int \partial_\mathrm{h}^{i}p\mathrm{div} \partial_\mathrm{h}^{i}\eta\mathrm{d}y
:=\|\partial_{\mathrm{h}}^{i} u\|_0^2
+\frac{1}{2}\int \partial_\mathrm{h}^{i}w\cdot\nabla \times  \partial_\mathrm{h}^{i}u\mathrm{d}y+\sum_{k=4}^6I_{k}.
\end{align}

Following the same argument as that used in \eqref{202604251942}--\eqref{202604251950}, we can estimate that
\begin{align}\label{202604251942nm}
|I_4|
\lesssim\begin{cases}
\sqrt{\mathcal{E}_{H}}{\mathcal{D}}_{L}
& \;\hbox{ for }\; 0\leqslant i\leqslant3;\\[1mm]
\sqrt{\mathcal{E}_{H}}{\mathcal{D}}_{H}+(\|u\|_3+\|(\nabla w,\nabla p)\|_{1})\|\eta\|_{6}^2
& \;\hbox{ for }\; 4\leqslant i\leqslant5.
\end{cases}\qquad\qquad
\end{align}
and
\begin{align}\label{202604251950nm}
|I_5|
&\lesssim
\begin{cases}
\|\mathcal{N}^2\|_{0}\|\nabla \times  \eta\|_{0}
& \hbox{for }\; i=0;\\[1mm]
\|\mathcal{N}^2\|_{1,0}\|\nabla \times  \eta\|_{2i-1,0}
& \hbox{for }\; 1\leqslant i\leqslant3;\\[1mm]
\|\mathcal{N}^2\|_{i-1,0}\|\nabla \times  \eta\|_{i+1,0}
& \hbox{for }\; i=4;\\[1mm]
\|\mathcal{N}^2\|_{5,0}\|\nabla \times  \eta\|_{5,0}
& \hbox{for }\; i=5
\end{cases}
\lesssim\begin{cases}
\sqrt{\mathcal{E}_{H}}{\mathcal{D}}_{L}
& \hbox{for }\; 0\leqslant i\leqslant3;\\[1mm]
\sqrt{\mathcal{E}_{H}}{\mathcal{D}}_{H}
& \hbox{for }\; i=4;\\[1mm]
\|(\eta,u)\|_{3}\|\eta\|_{6}\|(\eta,u)\|_{6}
& \hbox{for }\; i=5.
\end{cases}
\end{align}

It remains to estimate the term $I_6$.
First, it is known that (see \cite[(4.57)]{WYTIVNMI})
$$\int q\mathrm{div}\eta\mathrm{d}y=-\int\psi\cdot\nabla q\mathrm{d}y
\lesssim\|\eta\|_2^2\|\nabla q\|_0,$$
where
$$  \psi:=-\left(\begin{array}{c}
          \eta_1(\partial_2\eta_2+\partial_3\eta_3 )-
           \eta_1(\partial_2\eta_3 \partial_3\eta_2
- \partial_2\eta_2\partial_3\eta_3) \\[1mm]
         \eta_2\partial_3\eta_3-\eta_1\partial_1\eta_2-
        \eta_1( \partial_1\eta_2 \partial_3\eta_3
-\partial_1\eta_3\partial_3\eta_2)\\[1mm]
         -\eta_1\partial_1\eta_3
-\eta_2\partial_2\eta_3-\eta_1(\partial_1\eta_3 \partial_2\eta_2
-\partial_1\eta_2\partial_2\eta_3)
        \end{array}\right)
        \quad\mbox{and}\;\;\mathrm{div}\eta=\mathrm{div}\psi.
$$
Next, utilizing \eqref{11202054} and integrating by parts in $y_{\mathrm{h}}$ when $i\geqslant3$, we can have the bound:
\begin{align*}
|I_6|
&\lesssim
\begin{cases}
\|\nabla p\|_{3}\|\mathrm{div}\eta\|_{2}
& \hbox{for }\; 1\leqslant i\leqslant3;\\[1mm]
\|\nabla p\|_{i-1}\|\mathrm{div}\eta\|_{i}
& \hbox{for }\; i=4;\\[1mm]
\|\nabla p\|_{4}\|\mathrm{div}\eta\|_{5}
& \hbox{for }\; i=5
\end{cases}
\lesssim\begin{cases}
\sqrt{\mathcal{E}_{H}}{\mathcal{D}}_{L}
& \hbox{for }\; 1\leqslant i\leqslant3;\\[1mm]
\sqrt{\mathcal{E}_{H}}{\mathcal{D}}_{H}
& \hbox{for }\; i=4;\\[1mm]
\|\eta\|_{3}\|\eta\|_{6}\|\nabla p\|_{4}
& \hbox{for }\; i=5.
\end{cases}
\end{align*}
Therefore,  we conclude that
\begin{align}\label{202604260901}
|I_6|
\lesssim\begin{cases}
\sqrt{\mathcal{E}_{H}}{\mathcal{D}}_{L}
& \hbox{for }\; 0\leqslant i\leqslant3;\\[1mm]
\sqrt{\mathcal{E}_{H}}{\mathcal{D}}_{H}+\|\eta\|_{3}\|\eta\|_{6}\|\nabla p\|_{4}
& \hbox{for }\; 4\leqslant i\leqslant5.
\end{cases}
\end{align}
Consequently, plugging \eqref{202604251942nm}--\eqref{202604260901} into \eqref{202604260821} and invoking \eqref{friedrich} applied to $\partial_{\mathrm{h}}^{i} u$, then  \eqref{202109231700} follows.
This completes the proof.
\hfill$\Box$
\end{pf}

\subsection{Energy estimates for the temporal derivatives}

In this subsection, we derive the energy estimates for the temporal derivatives of $(u,w)$.

\begin{lem}\label{badiseqinM}
Under assumption \eqref{aprpioses} with sufficiently small $\delta$, there holds that
\begin{align}
&\label{11300905}
\frac{\mathrm{d}}{\mathrm{d}t}\left(\|(\partial_{t}^ju,\partial_{t}^jw)\|_{0}^2+\lambda\|(\bar{B}\cdot\nabla)\partial_{t}^{j-1}u\|_{0}^2\right)
+c\|(\nabla_{\mathcal{A}}\partial_{t}^ju,\partial_{t}^jw)\|_0^2\nonumber\\[1mm]
&\lesssim
\begin{cases}
\sqrt{\mathcal{E}_H} \mathcal{D}_L& \;\mathrm{  for }\;\;  j=1; \\[1mm]
\sqrt{\mathcal{E}_H} \mathcal{D}_H& \;\mathrm{ for }\;\; j=2.
\end{cases}\\[1mm]
&\label{11300910}
\frac{1}{2}\frac{\mathrm{d}}{\mathrm{d}t}
\left((\mu+\chi)\|\nabla_{\mathcal{A}}\partial_t^2u\|_0^2+4\chi\|\partial_{t}^2w\|_0^2-4\chi\int\nabla_{\mathcal{A}}\times \partial_{t}^2u\cdot\partial_{t}^2w\mathrm{d}y\right)
+c\|(\partial_t^3u,\partial_t^3w)\|^2_{0}\nonumber\\[1mm]
&\lesssim \|u_{t}\|_{2}^2+\sqrt{\mathcal{E}_H}  \mathcal{D}_H.
\end{align}
\end{lem}
\begin{pf}
Taking the inner product of \eqref{202609221247temporal}$_1$ with $\partial_{t}^ju$ in $L^2$,
integrating by parts over $\Omega$,
and using \eqref{202609221247temporal}$_3$--\eqref{202609221247temporal}$_4$, we obtain
\begin{align}\label{11301525}
&\frac{1}{2}\frac{\mathrm{d}}{\mathrm{d}t}\left(\|\partial_{t}^ju\|_{0}^2+\lambda\|(\bar{B}\cdot\nabla)\partial_{t}^{j-1}u\|_{0}^2\right)
+(\mu+\chi)\|\nabla_{\mathcal{A}}\partial_{t}^ju\|_0^2\nonumber\\[1mm]
&=2\chi\int\nabla_{\mathcal{A}}\times \partial_{t}^jw\cdot\partial_{t}^ju\mathrm{d}y
-\int\nabla\partial_{t}^jp\cdot D^{t,j}_u\mathrm{d}y+\int\mathcal{N}^{t,j}\cdot\partial_{t}^ju \mathrm{d}y.\quad\;
\end{align}
Similarly, taking the inner product of \eqref{202609221247temporal}$_2$ with $\partial_{t}^jw$ in $L^2$ yields
\begin{align}\label{11301535}
&\frac{1}{2}\frac{\mathrm{d}}{\mathrm{d}t}\|\partial_{t}^jw\|_{0}^2+4\chi\|\partial_{t}^jw\|_0^2
=2\chi\int\nabla_{\mathcal{A}}\times \partial_{t}^ju\cdot\partial_{t}^jw\mathrm{d}y+\int\mathcal{M}^{t,j}\cdot\partial_{t}^jw \mathrm{d}y.
\end{align}

Noting that $\epsilon_{kji}=-\epsilon_{ijk}$, we have
\begin{align}\label{11301535nm}
\epsilon_{ijk}\mathcal{A}_{jl}(\partial_{l}X_{k})Y_{i}
=\epsilon_{ijk}\mathcal{A}_{jl}\partial_{l}(X_{k}Y_{i})+\epsilon_{kji}\mathcal{A}_{jl}(\partial_{l}Y_{i})X_{k}.
\end{align}
Invoking \eqref{Akl=0}, the boundary condition $\partial_{t}^ju|_{\partial\Omega}=0$, and \eqref{11301535nm}, we readily deduce that
\begin{align}\label{202604271627}
\int\nabla_{\mathcal{A}}\times \partial_{t}^jw\cdot\partial_{t}^ju\mathrm{d}y=\int\nabla_{\mathcal{A}}\times \partial_{t}^ju\cdot\partial_{t}^jw\mathrm{d}y,\quad\mbox{for}\;1\leqslant j\leqslant3.
\end{align}
Consequently, combining \eqref{11301525} with \eqref{11301535} gives rise to
\begin{align}\label{113004251525}
&\frac{1}{2}\frac{\mathrm{d}}{\mathrm{d}t}\left(\|(\partial_{t}^ju,\partial_{t}^jw)\|_{0}^2+\lambda\|(\bar{B}\cdot\nabla)\partial_{t}^{j-1}u\|_{0}^2\right)
+(\mu+\chi)\|\nabla_{\mathcal{A}}\partial_{t}^ju\|_0^2+4\chi\|\partial_{t}^jw\|_0^2\nonumber\\[1mm]
&=4\chi\int\nabla_{\mathcal{A}}\times \partial_{t}^ju\cdot\partial_{t}^jw\mathrm{d}y
+\int\mathcal{N}^{t,j}\cdot\partial_{t}^ju \mathrm{d}y+\int\mathcal{M}^{t,j}\cdot\partial_{t}^jw \mathrm{d}y-\int\nabla\partial_{t}^jp\cdot D^{t,j}_u\mathrm{d}y.
\end{align}

Moreover, noting also that $$\Delta_{\mathcal{A}}\partial_{t}^ju=\nabla_{\mathcal{A}}\mathrm{div}_{\mathcal{A}}\partial_{t}^ju-\nabla_{\mathcal{A}}\times\nabla_{\mathcal{A}}\times\partial_{t}^ju
\quad\mbox{for}\;1\leqslant j\leqslant2,$$
and using the boundary condition $\partial_{t}^ju|_{\partial\Omega}=0$, \eqref{11301535nm},
as well as \eqref{Akl=0}, it readily follows that
\begin{align*}
\|\nabla_{\mathcal{A}}\times\partial_{t}^ju\|_{0}^2\leqslant\|\nabla_{\mathcal{A}}\partial_{t}^ju\|_{0}^2,
\end{align*}
which, together with Young's inequality, implies
\begin{align}\label{202604271630}
\|(\nabla_{\mathcal{A}}\partial_{t}^ju,\partial_{t}^jw)\|_0^2\lesssim
(\mu+\chi)\|\nabla_{\mathcal{A}}\partial_{t}^ju\|_0^2+4\chi\|\partial_{t}^jw\|_0^2-4\chi\int\nabla_{\mathcal{A}}\times \partial_{t}^ju\cdot\partial_{t}^jw\mathrm{d}y.
\end{align}
Substituting this into \eqref{113004251525}, we can refine to be
\begin{align}\label{04251135}
&\frac{1}{2}\frac{\mathrm{d}}{\mathrm{d}t}\left(\|(\partial_{t}^ju,\partial_{t}^jw)\|_{0}^2+\lambda\|\partial_{\bar{B}}\partial_{t}^{j-1}u\|_{0}^2\right)
+c\|(\nabla_{\mathcal{A}}\partial_{t}^ju,\partial_{t}^jw)\|_0^2\nonumber\\[1mm]
&\lesssim\int\mathcal{N}^{t,j}\cdot\partial_{t}^ju \mathrm{d}y+\int\mathcal{M}^{t,j}\cdot\partial_{t}^jw \mathrm{d}y-\int\nabla\partial_{t}^jp\cdot D^{t,j}_u\mathrm{d}y.
\end{align}
Next we proceed to estimate the integrals on the right hand side of \eqref{04251135}.
Indeed, by virtue of H\"older's inequality, \eqref{202108082010} and \eqref{badiseqin3nes1n0}, we can estimate that
\begin{align*}
&\int\mathcal{N}^{t,j}\cdot\partial_{t}^ju \mathrm{d}y+\int\mathcal{M}^{t,j}\cdot\partial_{t}^jw \mathrm{d}y
-\int\nabla\partial_{t}^jp\cdot D^{t,j}_u\mathrm{d}y\nonumber\\[1mm]
&\lesssim\|(\mathcal{N}^{t,j},\mathcal{M}^{t,j})\|_{0}\|(\partial_{t}^ju,\partial_{t}^jw)\|_{0}
+\|\nabla\partial_{t}^jp\|_{0}\|D^{t,j}_u\|_{0}
\lesssim
\begin{cases}
\sqrt{\mathcal{E}_H} \mathcal{D}_L& \;\hbox{  for }\;  j=1; \\[1mm]
\sqrt{\mathcal{E}_H} \mathcal{D}_H& \;\hbox{ for }\; j=2.
\end{cases}
\end{align*}
Inserting it  into \eqref{04251135}, then \eqref{11300905} follows.

We now turn to establish \eqref{11300910}.
Taking the inner product of \eqref{202609221247temporal}$_1$ with $j=2$ against $\partial_{t}^3u$ in $L^2$,
integrating by parts over $\Omega$,
and using \eqref{202609221247temporal}$_3$--\eqref{202609221247temporal}$_4$, we obtian
\begin{align}\label{113015250427}
&\frac{1}{2}\frac{\mathrm{d}}{\mathrm{d}t}(\mu+\chi)\|\nabla_{\mathcal{A}}\partial_{t}^2u\|_0^2+\|\partial_{t}^3u\|_{0}^2\nonumber\\[1mm]
&=\lambda\int(\bar{B}\cdot\nabla)^2u_t\cdot\partial_{t}^3u\mathrm{d}y
+2\chi\int\nabla_{\mathcal{A}}\times \partial_{t}^2w\cdot\partial_{t}^3u\mathrm{d}y\nonumber\\[1mm]
&\quad
+\int\mathcal{N}^{t,2}\cdot\partial_{t}^2u \mathrm{d}y-\int\nabla\partial_{t}^2p\cdot D^{t,3}_u\mathrm{d}y
+(\mu+\chi)\int\nabla_{\mathcal{A}_t}\partial_{t}^2u\cdot\partial_{t}^3u\mathrm{d}y.
\end{align}
Similarly, taking the inner product of \eqref{202609221247temporal}$_2$ with $j=2$ against $\partial_{t}^3w$ in $L^2$ gives
\begin{align}\label{113015350427}
&2\chi\frac{\mathrm{d}}{\mathrm{d}t}\|\partial_{t}^2w\|_{0}^2+\|\partial_{t}^3w\|_0^2
=2\chi\int\nabla_{\mathcal{A}}\times \partial_{t}^2u\cdot\partial_{t}^3w\mathrm{d}y+\int\mathcal{M}^{t,2}\cdot\partial_{t}^3w \mathrm{d}y.
\end{align}
Furthermore, in light of \eqref{202604271627}, we observe that
\begin{align}\label{202604271647}
&2\chi\int\nabla_{\mathcal{A}}\times \partial_{t}^2w\cdot\partial_{t}^3u\mathrm{d}y
=2\chi\int\nabla_{\mathcal{A}}\times \partial_{t}^3u\cdot\partial_{t}^2w\mathrm{d}y\nonumber\\[1mm]
&=2\chi\frac{\mathrm{d}}{\mathrm{d}t}\int\nabla_{\mathcal{A}}\times \partial_{t}^2u\cdot\partial_{t}^2w\mathrm{d}y
-2\chi\int\nabla_{\mathcal{A}}\times \partial_{t}^2u\cdot\partial_{t}^3w\mathrm{d}y
-2\chi\int\nabla_{\mathcal{A}_t}\times \partial_{t}^2u\cdot\partial_{t}^2w\mathrm{d}y.
\end{align}
Combining \eqref{113015250427}--\eqref{202604271647}, it follows that
\begin{align}\label{113009101653}
&\frac{1}{2}\frac{\mathrm{d}}{\mathrm{d}t}
\left((\mu+\chi)\|\nabla_{\mathcal{A}}\partial_t^2u\|_0^2+4\chi\|\partial_{t}^2w\|_0^2-4\chi\int\nabla_{\mathcal{A}}\times \partial_{t}^2u\cdot\partial_{t}^2w\mathrm{d}y\right)
+\|(\partial_{t}^3u,\partial_t^3w)\|_{0}^2\nonumber\\[1mm]
&=\lambda\int(\bar{B}\cdot\nabla)^2u_t\cdot\partial_{t}^3u\mathrm{d}y
+\int\mathcal{N}^{t,2}\cdot\partial_{t}^2u \mathrm{d}y+\int\mathcal{M}^{t,2}\cdot\partial_{t}^3w \mathrm{d}y
-\int\nabla\partial_{t}^2p\cdot D^{t,3}_u\mathrm{d}y\nonumber\\[1mm]
&\quad
+(\mu+\chi)\int\nabla_{\mathcal{A}_t}\partial_{t}^2u\cdot\partial_{t}^3u\mathrm{d}y-2\chi\int\nabla_{\mathcal{A}_t}\times \partial_{t}^2u\cdot\partial_{t}^2w\mathrm{d}y.
\end{align}
By virtue of H\"older's inequality, \eqref{prtislsafdsfs} and \eqref{badiseqin3nes1n0}, we can estimate that
\begin{align*}
&\lambda\int(\bar{B}\cdot\nabla)^2u_t\cdot\partial_{t}^3u\mathrm{d}y\lesssim\|u_{t}\|_2\|\partial_{t}^3u\|_0,\nonumber\\[1mm]
&\int\mathcal{N}^{t,2}\cdot\partial_{t}^2u \mathrm{d}y+
\int\mathcal{M}^{t,2}\cdot\partial_{t}^3w \mathrm{d}y
-\int\nabla\partial_{t}^2p\cdot D^{t,3}_u\mathrm{d}y\lesssim\sqrt{\mathcal{E}_H}  \mathcal{D}_H,\qquad\qquad\qquad\nonumber\\[1mm]
&(\mu+\chi)\int\nabla_{\mathcal{A}_t}\partial_{t}^2u\cdot\partial_{t}^3u\mathrm{d}y
-2\chi\int\nabla_{\mathcal{A}_t}\times \partial_{t}^2u\cdot\partial_{t}^2w\mathrm{d}y
\lesssim\sqrt{\mathcal{E}_H}  \mathcal{D}_H.
\end{align*}
Plugging the above three estimates into \eqref{113009101653} and using Young's inequality, yields \eqref{11300910}.
 This completes the proof.
\hfill$\Box$
\end{pf}

\subsection{Energy estimates on the normal derivatives}
In this subsection, we establish the estimates for the normal derivatives of $(\eta, u, w, p)$.
We first use the regularity theory of the Stokes problem to derive further estimates for $(\eta,u,p)$ as in \cite{WYJTIKCT,JFJSJMFMOSERT}.
Subsequently, we utilize the ODE structure of the linear perturbed system to recover the estimates for $w$, which simultaneously yields further estimates for
$(\bar{B}\cdot \nabla)^2\eta$ and $u$.

Denote $\varphi:=u+\lambda\bar{B}_3^2\eta/(\mu+\chi)$, we deduce from \eqref{202609221247} that $\varphi$ satisfies the following Stokes problem:
\begin{equation}\label{n0101nn928nenms}
\begin{cases}
-(\mu+\chi)\Delta \varphi+\nabla p= F\qquad &\mbox{in } \Omega ,\\[0.5mm]
\mathrm{div}\varphi =G &\mbox{in } \Omega ,\\[0.5mm]
\varphi=0 &\mbox{on } \partial\Omega,
\end{cases}
\end{equation}
where $F$ and $G$ are given by
\begin{align*}
&F:=\lambda\big((\bar{B}_{\mathrm{h}}\cdot \nabla_{\mathrm{h}})^2+2(\bar{B}_3\partial_3)(\bar{B}_{\mathrm{h}}\cdot \nabla_{\mathrm{h}})\big)\eta
-\lambda\bar{B}_3^2\Delta_{\mathrm{h}}\eta-u_t +2\chi \nabla \times w+\mathcal{N}^1,\\[1mm]
&G:=\lambda\bar{B}_3^2\mathrm{div}\eta/(\mu+\chi)+\mathcal{N}^3.
\end{align*}
Applying the horizontal derivatives $\partial_\mathrm{h}^{k}$ to the problem \eqref{n0101nn928nenms}, and invoking
the regularity theory of the Stokes problem (see Lemma \ref{10220835}) to the resulting problem, we deduce that, for $0\leqslant k\leqslant i$,
\begin{align}\label{11292247}
&\|\varphi \|_{k,i-k+2}^2+\|\nabla p \|_{k,i-k}^2\nonumber\\[1mm]
&\lesssim\|\varphi\|_{k,0}^2+\|F\|_{k,i-k}^2+\|G\|_{k,i-k+1}^2\nonumber\\[1mm]
&\lesssim\|\eta\|_{k+1,i-k+1}^2+\|(\eta,u)\|_{k,0}^2+\|(u_t,\nabla\times w)\|_{i}^2
+\|\mathcal{N}^1\|_{i}^2+\|(\mathrm{div}\eta,\mathcal{N}^3)\|_{i+1}^2.
\end{align}

We now use the above estimate to establish the estimates for $(\eta,u,p)$.

\begin{lem}\label{lem:11292030es}
Under assumption \eqref{aprpioses} with sufficiently small $\delta$, it holds that
\begin{align}
&\label{202604271129}
\frac{\mathrm{d}}{\mathrm{d}t}\overline{\|\eta\|}_{i+2}^2+\|(\eta,u)\|_{i+2}^2+\|\nabla p\|_{i}^2
\lesssim
\begin{cases}
\|(\eta,u)\|_{\underline{2},1}^2+\|(u_t,\nabla\times w)\|_{1}^2 & \;\;\mathrm{for }\;\;  i=1; \\[1mm]
\|(\eta,u)\|_{\underline{4},1}^2+\|(u_t,\nabla\times w)\|_{3}^2 & \;\;\mathrm{for }\;\; i=3; \\[1mm]
\mathcal{E}_H+\mathcal{D}_H & \;\;\mathrm{for }\;\; i=4,
\end{cases}
\end{align}
\end{lem}
where the functional $\overline{\|\eta\|}_{i+2}^2$ is equivalent to ${\|\eta\|}_{i+2}^2$.
\begin{pf}
Notice that $\bar{B}_3\neq0$ and
\begin{align*}
\| \varphi \|_{k,i-k+2}^2
=\frac{\mathrm{d}}{\mathrm{d}t}\left\|\sqrt{\bar{B}_3^2/(\mu+\chi)} \eta\right\|_{k,i-k+2}^2
+\|\bar{B}_3^2\eta/(\mu+\chi) \|_{k,i-k+2}^2+\| u\|_{k,i-k+2}^2.
\end{align*}
Substituting it into \eqref{11292247} then gives
\begin{align}\label{11292247n0}
&\frac{\mathrm{d}}{\mathrm{d}t}\left\|\sqrt{\bar{B}_3^2/(\mu+\chi)} \eta\right\|_{k,i-k+2}^2
+\|(\eta, u)\|_{k,i-k+2}^2+\|\nabla p \|_{k,i-k}^2\nonumber\\[1mm]
&\lesssim\|\eta\|_{k+1,i-k+1}^2+\|(\eta,u)\|_{k,0}^2+\|(u_t,\nabla\times w)\|_{i}^2+\|\mathcal{N}^1\|_{i}^2+\|(\mathrm{div}\eta,\mathcal{N}^3)\|_{i+1}^2.
 \end{align}
By virtue of the recursive inequality \eqref{11292247n0} from $k=0$ to $i$, we readily obtain
\begin{equation}\label{11300824}
\begin{aligned}
&\frac{\mathrm{d}}{\mathrm{d}t}
\overline{\|\eta\|}_{i+2}^2
+\|(\eta,u) \|_{i+2}^2
+\|\nabla p\|_{i}^2\\[1mm]
&\lesssim\|\eta\|_{i+1,1}^2+\|(\eta,u)\|_{\underline{i},0}^2+\|(u_t,\nabla\times w)\|_{i}^2
+\|\mathcal{N}^1\|_{i}^2+\|(\mathrm{div}\eta,\mathcal{N}^3)\|_{i+1}^2,
\end{aligned}
\end{equation}
where $$\overline{\|\eta\|}_{i+2}^2
:=\displaystyle\sum_{0\leqslant k\leqslant i}a_k\|\sqrt{\lambda\bar{B}_3^2/(\mu+\chi)} \eta\|_{k,i-k+2}^2,$$
and the coefficients $a_k$ are suitably large positive constants.
It is evident that $\overline{\|\eta\|}_{i+2}^2$ is equivalent to $\|\eta\|_{i+2}^2$.
Therefore, taking $i=1$, $3$, and $4$ in \eqref{11300824} respectively, and utilizing \eqref{11202054} together with
\eqref{06011711jumpv}--\eqref{11210836}, we immediately arrive at \eqref{202604271129} for sufficiently $\delta$. This completes the proof of Lemma \ref{lem:11292030es}.
\hfill$\Box$
\end{pf}

Moreover, applying the temporal derivatives $\partial_t^{j}$ to \eqref{202609221247}, we can find that $(\partial_t^ju,\partial_t^jp)$ satisfies
\begin{equation}\label{n0101nn928nes}
\begin{cases}
-(\mu+\chi)\Delta \partial_t^{j}u+\nabla \partial_t^{j}p= \partial_t^{j}H\quad &\mbox{in } \Omega ,\\[0.5mm]
\mathrm{div}\partial_t^{j} u=\partial_t^{j}\mathcal{N}^3 &\mbox{in } \Omega ,\\[0.5mm]
\partial_t^{j}u=0 &\mbox{on } \partial\Omega,
\end{cases}
\end{equation}
where $H$ is defined by
\begin{align}
& H:=\lambda(\bar{B}\cdot \nabla)^2\eta-u_t +2\chi \nabla \times w+\mathcal{N}^1.\nonumber
\end{align}
Similarly to \eqref{11292247}, applying the regularity theory of Stokes problem in Lemma \ref{10220835} to \eqref{n0101nn928nes}, yields, for $0\leqslant j\leqslant i$,
\begin{align}\label{112922471039}
&\|\partial_t^{j}u\|_{i-2j+1}^2+\|\nabla \partial_t^j p \|_{i-2j-1}^2\nonumber\\[1mm]
&\lesssim\|\partial_t^{j}u\|_{0}^2+\|\partial_t^{j}H\|_{i-2j-1}^2+\|\partial_t^{j}\mathcal{N}^3\|_{i-2j}^2\nonumber\\[1mm]
&\lesssim\|\partial_t^{j}u\|_{0}^2+\|\partial_t^{j}(u_t,\nabla\times w,(\bar{B}\cdot \nabla)^2\eta)\|_{i-2j-1}^2
+\|\partial_t^{j}\mathcal{N}^1\|_{i-2j-1}^2+\|\partial_t^{j}\mathcal{N}^3\|_{i-2j}^2.
\end{align}

We now establish the following estimates for $(u,w,p)$.
\begin{lem}\label{lem:0426}
Under assumption \eqref{aprpioses} with sufficiently small $\delta$, there holds that
\begin{align}
&\label{202604271340}
\|u\|_{2}^2+\|\nabla p\|_{0}^2\lesssim\|\eta\|_{2}^2+\|(\nabla w, u,u_t)\|_{0}^2,\\[1.5mm]
&\label{202604271303}
\|u\|_{6}^2+\|\nabla p\|_{4}^2
\lesssim \|u\|_{0}+\|((\bar{B}\cdot \nabla)^2\eta,\nabla w,u_t)\|_{4}^2,\\ 
&\label{202604271345}
\sum_{j=0}^{1}\left(\|\partial_t^ju\|_{5-2j}^2+\|\nabla\partial_t^jp\|_{3-2j}^2\right)
\lesssim\|((\bar{B}\cdot \nabla)^2\eta,\nabla w)\|_{3}^2 +\|u_{tt}\|_{1}^2 +\sum_{j=0}^{1}\|\partial_t^{j}u\|_{0}^2+(\mathcal{E}_H)^2,\\
&\label{202604271347}
\sum_{j=1}^{2}\left(\|\partial_t^ju\|_{6-2j}^2+\|\nabla\partial_t^j p\|_{4-2j}^2\right)
\lesssim \sum_{j=0}^{3}\|\partial_t^{j}u\|_{0}^2+\sum_{j=1}^{2}\|\partial_t^jw\|_{5-2j}^2+\mathcal{E}_{H}\mathcal{D}_H,\\
&\label{2026042712136}
\sum_{j=1}^{2}\|\partial_t^jw\|_{6+k-2j}^2\lesssim\|w\|_{4+k}^2
+\sum_{j=0}^{1}\left(\|\partial_t^ju\|_{5+k-2j}^2+\|\nabla\partial_t^jp\|_{3+k-2j}^2\right)\quad\mbox{for}\;\;k=0,1.
\end{align}
\end{lem}
\begin{pf}
To begin with, we derive the estimates \eqref{202604271340} and \eqref{202604271303}.
Taking $(j,i)=(0,1)$ in \eqref{112922471039} and using \eqref{06011711jumpv}--\eqref{11210836}, we have
\begin{align*}
\|u\|_{2}^2+\|\nabla p \|_{0}^2
&\lesssim\|u\|_{0}^2+\|(u_t,\nabla\times w,(\bar{B}\cdot \nabla)^2\eta)\|_{0}^2
+\|\mathcal{N}^1\|_{0}^2+\|\mathcal{N}^3\|_{1}^2\nonumber\\[1.5mm]
&\lesssim\|u\|_{0}^2+\|(u_t,\nabla w,(\bar{B}\cdot \nabla)^2\eta)\|_{0}^2+\|\eta\|_{3}^2(\|u\|_{2}^2+\|(\nabla p,\nabla w) \|_{0}^2).
\end{align*}
This implies \eqref{202604271340} for sufficiently small $\delta$.

In the same way, taking $(j,i)=(0,5)$ in \eqref{112922471039} and using \eqref{06011711jumpv}--\eqref{11210836}, we obtain
\begin{align*}
\|u\|_{6}^2+\|\nabla p \|_{4}^2
&\lesssim\|u\|_{0}^2+\|(u_t,\nabla\times w,(\bar{B}\cdot \nabla)^2\eta)\|_{4}^2
+\|\mathcal{N}^1\|_{4}^2+\|\mathcal{N}^3\|_{5}^2\nonumber\\[1.5mm]
&\lesssim\|u\|_{0}^2+\|(u_t,\nabla w,(\bar{B}\cdot \nabla)^2\eta)\|_{4}^2+\|\eta\|_{6}^2(\|u\|_{6}^2+\|(\nabla p,\nabla w) \|_{4}^2),
\end{align*}
which in turn yields \eqref{202604271303} for sufficiently small $\delta$.

We now turn to the estimate of \eqref{202604271345}.
Invoking the recursive inequality \eqref{112922471039} for $i=4$ from $j=0$ to $1$, we have
\begin{align}\label{202604272045}
\sum_{j=0}^{1}\left(\|\partial_t^ju\|_{5-2j}^2+\|\nabla\partial_t^jp\|_{3-2j}^2\right)
&\lesssim\|(u,u_t)\|_{0}^2+\|(\nabla w,(\bar{B}\cdot \nabla)^2\eta)\|_{3}^2+\|(\nabla w_t,(\bar{B}\cdot \nabla)^2 u,u_{tt})\|_{1}^2\nonumber\\
&\;\quad+\sum_{j=0}^{1}\left(\|\partial_t^j\mathcal{N}^1\|_{3-2j}^2+\|\partial_t^j\mathcal{N}^3\|_{4-2j}^2\right).
\end{align}
Furthermore, by virtue of \eqref{202609221247n}$_3$, \eqref{06011711jumpv}--\eqref{2021081110937}  and \eqref{10202100m}, we can estimate that
\begin{align*}
\|\nabla w_t\|_{1}^2+\sum_{j=0}^{1}\left(\|\partial_t^j\mathcal{N}^1\|_{3-2j}^2+\|\partial_t^j\mathcal{N}^3\|_{4-2j}^2\right)\lesssim
\|(\nabla^2 u,\nabla w)\|_{1}^2+(\mathcal{E}_{H})^2.
\end{align*}
Inserting the above estimate into \eqref{202604272045} and using the interpolation inequality leads to \eqref{202604271345}.

Analogously, we deduce from the recursive inequality \eqref{112922471039} for $i=5$ from $j=1$ to $2$ that
\begin{align*}
&\sum_{j=1}^{2}\left(\|\partial_t^ju\|_{6-2j}^2+\|\nabla\partial_t^j p\|_{4-2j}^2\right)\\
&\lesssim  \sum_{j=1}^{3}\|\partial_t^{j}u\|_{0}^2
+\sum_{j=1}^{2}\left(\|\partial_t^j\mathcal{N}^1\|_{4-2j}^2+\|\partial_t^j\mathcal{N}^3\|_{5-2j}^2\right)+\sum_{j=1}^{2}\|\partial_t^jw\|_{5-2j}^2.
\end{align*}
Combining this with \eqref{202108110942nm}--\eqref{10202100m}  then gives \eqref{202604271347}.

Finally, we establish the estimate \eqref{2026042712136}.
Indeed, it directly follows from \eqref{202609221247n}$_3$ that
\begin{align}\label{2026042712136m}
\sum_{j=1}^{2}\|\partial_t^jw\|_{6+k-2j}^2
&\lesssim\sum_{j=0}^{1}\|\partial_t^jw\|_{4+k-2j}^2
+\sum_{j=0}^{1}\|\partial_t^ju\|_{5+k-2j}^2+\sum_{j=0}^{1}\|\partial_t^j\mathcal{N}^2\|_{4+k-2j}^2\nonumber\\
&\lesssim\|w\|_{4+k}^2+\sum_{j=0}^{1}\|\partial_t^ju\|_{5+k-2j}^2+\sum_{j=0}^{1}\|\partial_t^j\mathcal{N}^2\|_{4+k-2j}^2.
\end{align}
In view of \eqref{11210836} and \eqref{10202100m}, it is easy to estimate that
$$\sum_{j=0}^{1}\|\partial_t^j\mathcal{N}^2\|_{4+k-2j}^2
\lesssim\mathcal{E}_{H}\sum_{j=0}^{1}\left(\|\partial_t^ju\|_{5+k-2j}^2+\|\nabla\partial_t^jp\|_{3+k-2j}^2\right).$$
Substituting this bound into \eqref{2026042712136m} yields \eqref{2026042712136} for sufficiently small $\delta$.
This completes the proof of Lemma \ref{lem:0426}.
\hfill$\Box$
\end{pf}

In the spirit of the approach in \cite{FengSCM},
we now utilize the ODE structure of the linear perturbed system to recover the estimates for the normal derivatives of $((\bar{B}\cdot \nabla)^2\eta,u,w)$.
We first denote
$$W:=(W_1,W_2,W_3)^{\top}=\nabla\times w.$$
Applying the curl operator to \eqref{202609221247n}$_3$ and using the identity $\Delta u= \nabla\mathrm{div}u-\nabla\times\nabla\times u$, one finds that
\begin{align}\label{202604101338}
W_{t}+4\chi W=2\chi(\nabla\mathrm{div}u-\Delta u)+\nabla\times\mathcal{N}^2.
\end{align}
In view of \eqref{202609221247n}$_2$ and \eqref{202604101338}, we observe that $\big((\bar{B}\cdot \nabla)^2\eta_{\mathrm{h}},\partial_3^2 u_{\mathrm{h}},W_{\mathrm{h}}\big)$ satisfies
\begin{equation}\label{2026195637es}
\begin{cases}
\lambda (\bar{B}\cdot \nabla)^2\eta_{\mathrm{h}}+(\mu + \chi)\partial_{3}^2u_{\mathrm{h}}+2\chi W_{\mathrm{h}}
=\partial_tu_{\mathrm{h}}  +  \nabla_{\mathrm{h}} p- (\mu + \chi) \Delta_{\mathrm{h}} u_{\mathrm{h}}
-\mathcal{N}^1_{\mathrm{h}}:=\tilde{\mathcal{N}}_{\mathrm{h}}
 &\mbox{in } \Omega ,\\[1mm]
\partial_{t}W_{\mathrm{h}}+4\chi W_{\mathrm{h}}+2\chi\partial_{3}^2u_{\mathrm{h}}
=2\chi(\nabla_{\mathrm{h}}\mathcal{N}^3-\Delta_{\mathrm{h}} u_{\mathrm{h}})+(\nabla\times\mathcal{N}^2)_{\mathrm{h}}:=\tilde{\mathcal{M}}_{\mathrm{h}} &\mbox{in } \Omega,
\end{cases}
\end{equation}
and that $W_{3}$ satisfies
\begin{equation}\label{2026195638es}
\partial_{t}W_{3}+4\chi W_{3}
=2\chi\big(\partial_3\mathrm{div}_{\mathrm{h}}u_{\mathrm{h}}-\Delta_{\mathrm{h}} u_{3}\big)+(\nabla\times\mathcal{N}^2)_{3}:=\tilde{\mathcal{M}}_{3}\quad \mbox{in } \Omega.
\end{equation}
Noting that the order of $\partial_3$ in the linear parts on the right hand sides of \eqref{2026195637es}--\eqref{2026195638es}
is lower than that of $\partial_3$ on the left hand side, this feature allows us to convert the
$y_3$-derivative estimates of the solution into the $y_{\mathrm{h}}$-derivative estimates.
More precisely, we have the following estimates for $(\bar{B}\cdot \nabla)^2\eta$, $u$ and $W$.
\begin{lem}\label{uwnormal}
Under assumption \eqref{aprpioses} with sufficiently small $\delta$, there holds that, 
\begin{align}\label{202604102028}
&\frac{\mathrm{d}}{\mathrm{d}t}\mathcal{H}_{i}(\eta,W)+\|((\bar{B}\cdot \nabla)^2\eta, W)\|_{i}^2+\|u\|_{i+2}^2\nonumber\\[1mm]
&\lesssim\begin{cases}
\|\nabla p\|_{1}\|((\bar{B}\cdot \nabla)\eta,\nabla u)\|_{2}^{1/2}\|((\bar{B}\cdot \nabla)\eta,\nabla u)\|_{\underline{2},0}^{1/2}
\\[0.5mm]
+\|((\bar{B}\cdot \nabla)\eta, \nabla u)\|_{\underline{2},0}^2+\|u_t\|_1^2
+\sqrt{\mathcal{E}_H} \mathcal{D}_L& \;\;\mathrm{for }\;\;  i=1; \\[2mm]
\|\nabla p\|_{4}\|((\bar{B}\cdot \nabla)\eta,\nabla u)\|_{5}^{1/2}((\bar{B}\cdot \nabla)\eta,\nabla u)\|_{\underline{5},0}^{1/2}\\[0.5mm]
+\|((\bar{B}\cdot \nabla)\eta, \nabla u)\|_{\underline{5},0}^2+\|u_t\|_4^2+\sqrt{\mathcal{E}_H} \mathcal{D}_H+\|(\eta,u)\|_{3}\|\eta\|_{6}^2\quad& \;\;\mathrm{for }\;\; i=4,
\end{cases}
\end{align}
where the functional $\mathcal{H}_{i}(\eta,W)$ satisfies that
\begin{align}\label{202604261336}
\|((\bar{B}\cdot \nabla)^2\eta_{\mathrm{h}} ,W)\|_{i}^2\lesssim\mathcal{H}_{i}(\eta,W).
\end{align}
\end{lem}
\begin{pf}
The proof Lemma \ref{uwnormal} is divided into the following four steps.

{\bf{Step 1. Estimates for $\partial_3^2u_{\mathrm{h}}$ and $W_{\mathrm{h}}$.}}
Let $0\leqslant j\leqslant i\leqslant5$ and $0\leqslant k\leqslant i-j$.
Applying $\partial_{\mathrm{h}}^{j+k}\partial_{3}^{i-j-k}$ to \eqref{2026195637es}$_1$ and \eqref{2026195637es}$_2$,
and taking the inner product of
the resulting identities with $(\bar{B}\cdot \nabla)^2\partial_{\mathrm{h}}^{j+k}\partial_{3}^{i-j-k}u_{\mathrm{h}}$ and $\bar{B}_3^2\partial_{\mathrm{h}}^{j+k}\partial_{3}^{i-j-k} W_{\mathrm{h}}$ in $L^2$, respectively, we have
\begin{align}\label{2092650411342}
&\frac{\lambda}{2}\frac{\mathrm{d}}{\mathrm{d}t}\|(\bar{B}\cdot \nabla)^2\partial_{\mathrm{h}}^{j+k}\partial_{3}^{i-j-k}\eta_{\mathrm{h}}\|_{0}^2
\nonumber\\[1mm]&
+\bar{B}_3^2(\mu+\chi)\|\partial_{\mathrm{h}}^{j+k}\partial_{3}^{2+i-j-k}u_{\mathrm{h}}\|_{3}^2
+2\chi\bar{B}_3^2\int\partial_{\mathrm{h}}^{j+k}\partial_{3}^{i-j-k}W_{\mathrm{h}}\cdot\partial_{\mathrm{h}}^{j+k}\partial_{3}^{2+i-j-k}u_{\mathrm{h}}\mathrm{d}y\nonumber\\[1mm]&
=\int\partial_{\mathrm{h}}^{j+k}\partial_{3}^{i-j-k}\tilde{\mathcal{N}}_{\mathrm{h}}
\cdot(\bar{B}\cdot \nabla)^2\partial_{\mathrm{h}}^{j+k}\partial_{3}^{i-j-k}u_{\mathrm{h}}\mathrm{d}y\nonumber\\[1mm]&
\quad-(\mu+\chi)\int\partial_{\mathrm{h}}^{j+k}\partial_{3}^{2+i-j-k}u_{\mathrm{h}}\cdot
\big((\bar{B}_{\mathrm{h}}\cdot \nabla_{\mathrm{h}})^2+2(\bar{B}_3\partial_3)(\bar{B}_{\mathrm{h}}\cdot \nabla_{\mathrm{h}})\big)\partial_{\mathrm{h}}^{j+k}\partial_{3}^{i-j-k}u_{\mathrm{h}}\mathrm{d}y\nonumber\\[1mm]&
\quad-2\chi\int\partial_{\mathrm{h}}^{j+k}\partial_{3}^{i-j-k}W_{\mathrm{h}}\cdot
\big((\bar{B}_{\mathrm{h}}\cdot \nabla_{\mathrm{h}})^2+2(\bar{B}_3\partial_3)(\bar{B}_{\mathrm{h}}\cdot \nabla_{\mathrm{h}})\big)\partial_{\mathrm{h}}^{j+k}\partial_{3}^{i-j-k}u_{\mathrm{h}}\mathrm{d}y
\end{align}
and
\begin{align}\label{2092650411340}
&\frac{\bar{B}_3^2}{2}\frac{\mathrm{d}}{\mathrm{d}t}\|\partial_{\mathrm{h}}^{j+k}\partial_{3}^{i-j-k} W_{\mathrm{h}}\|_{0}^2
\nonumber\\[1mm]&
+4\chi\bar{B}_3^2\|\partial_{\mathrm{h}}^{j+k}\partial_{3}^{i-j-k} W_{\mathrm{h}}\|_{0}^2+2\chi\bar{B}_3^2\int\partial_{\mathrm{h}}^{j+k}\partial_{3}^{2+i-j-k}u_{\mathrm{h}}\cdot\partial_{\mathrm{h}}^{j+k}\partial_{3}^{i-j-k}W_{\mathrm{h}}\mathrm{d}y
\qquad\quad\nonumber\\[1mm]&
=\bar{B}_3^2\int\partial_{\mathrm{h}}^{j+k}\partial_{3}^{i-j-k}\tilde{\mathcal{M}}_{\mathrm{h}}
\cdot\partial_{\mathrm{h}}^{j+k}\partial_{3}^{i-j-k}(\bar{B}\cdot \nabla)^2u_{\mathrm{h}}\mathrm{d}y.
\end{align}
Combining \eqref{2092650411340} with \eqref{2092650411342}, it follows that
\begin{align}\label{2092650411345}
&\frac{1}{2}\frac{\mathrm{d}}{\mathrm{d}t}
\left(\lambda\|(\bar{B}\cdot \nabla)^2\partial_{\mathrm{h}}^{j+k}\partial_{3}^{i-j-k}\eta_{\mathrm{h}}\|_{0}^2
+\bar{B}_3^2\|\partial_{\mathrm{h}}^{j+k}\partial_{3}^{i-j-k} W_{\mathrm{h}}\|_{0}^2\right)
+(\mu+\chi)\bar{B}_3^2\|\partial_{\mathrm{h}}^{j+k}\partial_{3}^{2+i-j-k}u_{\mathrm{h}}\|_{3}^2\nonumber\\[1mm]&
+4\chi\bar{B}_3^2\|\partial_{\mathrm{h}}^{j+k}\partial_{3}^{i-j-k} W_{\mathrm{h}}\|_{0}^2
+4\chi\bar{B}_3^2\int\partial_{\mathrm{h}}^{j+k}\partial_{3}^{i-j-k}W_{\mathrm{h}}\cdot\partial_{\mathrm{h}}^{j+k}\partial_{3}^{2+i-j-k}u_{\mathrm{h}}\mathrm{d}y
=\sum_{l=7}^{10}I_{l},
\end{align}
where
\begin{align*}
&I_7:=\int\partial_{\mathrm{h}}^{j+k}\partial_{3}^{i-j-k}\tilde{\mathcal{N}}_{\mathrm{h}}
\cdot(\bar{B}\cdot \nabla)^2\partial_{\mathrm{h}}^{j+k}\partial_{3}^{i-j-k}u_{\mathrm{h}}\mathrm{d}y,\\[1mm]
&I_8:=\bar{B}_3^2\int\partial_{\mathrm{h}}^{j+k}\partial_{3}^{i-j-k}\tilde{\mathcal{M}}_{\mathrm{h}}
\cdot(\bar{B}\cdot \nabla)^2\partial_{\mathrm{h}}^{j+k}\partial_{3}^{i-j-k}u_{\mathrm{h}}\mathrm{d}y,\\[1mm]
&I_9:=-(\mu+\chi)\int\partial_{\mathrm{h}}^{j+k}\partial_{3}^{2+i-j-k}u_{\mathrm{h}}\cdot
\big((\bar{B}_{\mathrm{h}}\cdot \nabla_{\mathrm{h}})^2+2(\bar{B}_3\partial_3)(\bar{B}_{\mathrm{h}}\cdot \nabla_{\mathrm{h}})\big)\partial_{\mathrm{h}}^{j+k}\partial_{3}^{i-j-k}u_{\mathrm{h}}\mathrm{d}y,\\[1mm]
&I_{10}:=-2\chi\int\partial_{\mathrm{h}}^{j+k}\partial_{3}^{i-j-k}W_{\mathrm{h}}\cdot
\big((\bar{B}_{\mathrm{h}}\cdot \nabla_{\mathrm{h}})^2+2(\bar{B}_3\partial_3)(\bar{B}_{\mathrm{h}}\cdot \nabla_{\mathrm{h}})\big)\partial_{\mathrm{h}}^{j+k}\partial_{3}^{i-j-k}u_{\mathrm{h}}\mathrm{d}y.
\end{align*}

Observe the identity
\begin{align*}
&\|\partial_{\mathrm{h}}^{j}\partial_{3}^{3-i}u_{\mathrm{h}}\|_0^2
+4 \int\partial_{\mathrm{h}}^{j+k}\partial_{3}^{i-j-k}W_{\mathrm{h}}\cdot\partial_{\mathrm{h}}^{j+k}\partial_{3}^{2+i-j-k}u_{\mathrm{h}}\mathrm{d}x
+4 \|\partial_{\mathrm{h}}^{j+k}\partial_{3}^{i-j-k}W_{\mathrm{h}}\|_0^2\qquad\nonumber\\[1.5mm]
&=\|\partial_{\mathrm{h}}^{j+k}\partial_{3}^{2+i-j-k}u_{\mathrm{h}}+2 \partial_{\mathrm{h}}^{j+k}\partial_{3}^{i-j-k}u_{\mathrm{h}}W_{\mathrm{h}}\|_{0}^2
\end{align*}
and the inequality
\begin{align*}
\|\partial_{\mathrm{h}}^{j+k}\partial_{3}^{i-j-k}W_{\mathrm{h}}\|_{0}^2\lesssim
\|\partial_{\mathrm{h}}^{j+k}\partial_{3}^{2+i-j-k}u_{\mathrm{h}}+2 \partial_{\mathrm{h}}^{j+k}\partial_{3}^{i-j-k}u_{\mathrm{h}}W_{\mathrm{h}}\|_{0}^2
+\epsilon \|\partial_{\mathrm{h}}^{j+k}\partial_{3}^{2+i-j-k}u_{\mathrm{h}}\|_{0}^2,
\end{align*}
which holds for any positive constant $\epsilon>0$.
Moreover, we can have the following bound:
\begin{align*}
\sum_{l=7}^{10}I_{l}
\lesssim&\|(\partial_{\mathrm{h}}^{j+k}\partial_{3}^{2+i-j-k}u_{\mathrm{h}},\partial_{\mathrm{h}}^{j+k}\partial_{3}^{i-j-k}W_{\mathrm{h}})\|_{0}
\big(\|u_t\|_{i}+\|\nabla u_{\mathrm{h}}\|_{j+1,i-j}\big)\qquad\nonumber\\[1mm]
&\;+\|(\nabla^2u_{\mathrm{h}},W_{\mathrm{h}})\|_{j,i-j}\|(\mathcal{N}^1,\nabla_{\mathrm{h}}\mathcal{N}^3, \nabla\times\mathcal{N}^2)\|_{i}\nonumber\\[1mm]
&\;+\int\partial_{\mathrm{h}}^{j+k}\partial_{3}^{i-j-k}\nabla_{\mathrm{h}} p
\cdot(\bar{B}\cdot \nabla)^2\partial_{\mathrm{h}}^{j+k}\partial_{3}^{i-j-k}u_{\mathrm{h}}\mathrm{d}y.
\end{align*}
Therefore, exploiting Young's inequality and the above three estimates, and summing over $k$, we deduce from  \eqref{2092650411345} that there exists a positive constant $c_1>0$ such that
\begin{align}\label{2092650411535}
&\frac{1}{2}\frac{\mathrm{d}}{\mathrm{d}t}\left(\lambda\|(\bar{B}\cdot \nabla)^2\eta_{\mathrm{h}}\|_{j,i-j}^2
+\bar{B}_3^2\| W_{\mathrm{h}}\|_{j,i-j}^2\right)
+c_1\big(\|u_{\mathrm{h}}\|_{j,i-j+2}^2+\|W_{\mathrm{h}}\|_{j,i-j}^2\big)
\nonumber\\[1mm]
&\lesssim
\|u_{\mathrm{h}}\|_{\underline{j+1},i-j+1}^2+\|u_t\|_{i}^2+\|\mathcal{N}^1\|_{i}^2+\|(\mathcal{N}^2,\mathcal{N}^3)\|_{i+1}^2\nonumber\\[1mm]
&\quad+\sum_{k=0}^{i-j}\left|\int\partial_{\mathrm{h}}^{j+k}\partial_{3}^{i-j-k}\nabla_{\mathrm{h}} p
\cdot\partial_{\mathrm{h}}^{j+k}\partial_{3}^{i-j-k}(\bar{B}\cdot \nabla)^2u_{\mathrm{h}}\mathrm{d}y\right|:=I_{11}.
\end{align}
It remains to estimate $I_{11}$, which we split into three cases.

For the case $i=0$, by using the interpolation inequality, we find that
\begin{align}\label{202604261537a}
I_{11}\lesssim\|\nabla p\|_{0}\|\nabla u_{\mathrm{h}}\|_{1}\lesssim\|\nabla p\|_{0}\|\nabla u_{\mathrm{h}}\|_{2}^{1/2}\|\nabla u_{\mathrm{h}}\|_{0}^{1/2}.
\end{align}
For the case $i\neq0$ and $j+k<i$, we can have the bound:
\begin{align}\label{202604261537b}
\left|\int\partial_{\mathrm{h}}^{j+k}\partial_{3}^{i-j-k}\nabla_{\mathrm{h}} p
\cdot(\bar{B}\cdot \nabla)^2\partial_{\mathrm{h}}^{j+k}\partial_{3}^{i-j-k}u_{\mathrm{h}}\mathrm{d}y\right|
\lesssim\|\nabla \partial_{\mathrm{h}}p\|_{j,i-j-1}\|u_{\mathrm{h}}\|_{j,i-j+2}.
\end{align}
For the remaining case $i\neq0$ and $j+k=i$, an integration by parts over $\Omega$ yields
\begin{align}\label{2092650411539nnmm}
&\int\partial_{\mathrm{h}}^{j+k}\partial_{3}^{i-j-k}\nabla_{\mathrm{h}} p
\cdot(\bar{B}\cdot \nabla)^2\partial_{\mathrm{h}}^{j+k}\partial_{3}^{i-j-k}u_{\mathrm{h}}\mathrm{d}y
\nonumber\\[1mm]&
=\int\nabla_{\mathrm{h}}\partial_{\mathrm{h}}^{i} p
\cdot(\bar{B}\cdot \nabla)^2\partial_{\mathrm{h}}^{i}u_{\mathrm{h}}\mathrm{d}y
\nonumber\\[1mm]
&=\int (\bar{B}\cdot \nabla)\partial_{\mathrm{h}}^{i}p\cdot(\bar{B}\cdot \nabla)\partial_{\mathrm{h}}^{i}\mathrm{div}_{\mathrm{h}}u_{\mathrm{h}}\mathrm{d}y
+\bar{B}_3\int_{\partial\Omega}\nabla_{\mathrm{h}}\partial_{\mathrm{h}}^{i}p\cdot(\bar{B}\cdot \nabla)\partial_{\mathrm{h}}^{i}u_{\mathrm{h}}\vec{n}_{3}\mathrm{d}y_{\mathrm{h}},
\end{align}
where $\vec{n}_{3}$ denotes the third component of the unit outward normal vector $\vec{n}$ on $\partial\Omega$, namely, $\vec{n}_{3}=1$ on $\mathbb{R}^2\times\{1\}$ and $\vec{n}_{3}=-1$ on $\mathbb{R}^2\times\{0\}$.
Exploiting the dual estimate \eqref{11190840} and the trace estimate \eqref{37190928}, we have the following bound
\begin{align}\label{20926504111737}
&\bigg|\bar{B}_3\int_{\partial\Omega}\nabla_{\mathrm{h}}\partial_{\mathrm{h}}^{i}p\cdot(\bar{B}\cdot \nabla)\partial_{\mathrm{h}}^{i}u_{\mathrm{h}}\vec{n}_{3}\mathrm{d}y_{\mathrm{h}}\bigg|\nonumber\\[1mm]
&\lesssim
|\partial_{\mathrm{h}}^{i}p|_{H^{1/2}(\partial\Omega)}|(\bar{B}\cdot \nabla)\partial_{\mathrm{h}}^{i}u_{\mathrm{h}}|_{H^{1/2}(\partial\Omega)}
\lesssim\|\partial_{\mathrm{h}}^{i}p\|_{1}\|(\bar{B}\cdot \nabla)\partial_{\mathrm{h}}^{i}u_{\mathrm{h}}\|_{\underline{1},0}^{1/2}\|(\bar{B}\cdot \nabla)\partial_{\mathrm{h}}^{i}u_{\mathrm{h}}\|_{1}^{1/2}\nonumber\\[1mm]
&\lesssim\|\nabla p\|_{\underline{i},0}\|\nabla u_{\mathrm{h}}\|_{\underline{i},1}^{1/2}\|\nabla u_{\mathrm{h}}\|_{\underline{i+1},0}^{1/2}.
\end{align}
Inserting \eqref{20926504111737} into \eqref{2092650411539nnmm} yields
\begin{align}\label{20260411na}
\bigg|\int\partial_{\mathrm{h}}^{j+k}\partial_{3}^{i-j-k}\nabla_{\mathrm{h}} p
\cdot(\bar{B}\cdot \nabla)^2\partial_{\mathrm{h}}^{j+k}\partial_{3}^{i-j-k}u_{\mathrm{h}}\mathrm{d}y\bigg|
\lesssim\|\nabla p\|_{i}\|\nabla u_{\mathrm{h}}\|_{i+1}^{1/2}\|\nabla u_{\mathrm{h}}\|_{\underline{i+1},0}^{1/2}.
\end{align}
Collecting \eqref{202604261537a}, \eqref{202604261537b} and \eqref{20260411na} leads to
\begin{align}\label{202604261537c}
I_{11}\lesssim(i-j)\|\nabla p\|_{j+1,i-j-1}\|u_{\mathrm{h}}\|_{j,i-j+2}+\|\nabla p\|_{i}(\|\nabla u_{\mathrm{h}}\|_{i+1}^{1/2}+\|\nabla u_{\mathrm{h}}\|_2^{1/2})
\|\nabla u_{\mathrm{h}}\|_{\underline{i+1},0}^{1/2}.
\end{align}
Therefore, plugging \eqref{202604261537c} into \eqref{2092650411535} and using Young's inequality, we deduce that
\begin{align}\label{2092650411635}
&\frac{1}{2}\frac{\mathrm{d}}{\mathrm{d}t}\left(\lambda\|(\bar{B}\cdot \nabla)^2\eta_{\mathrm{h}}\|_{j,i-j}^2
+\bar{B}_3^2\| W_{\mathrm{h}}\|_{j,i-j}^2\right)
+\frac{c_1}{2}\big(\|u_{\mathrm{h}}\|_{j,i-j+2}^2+\|W_{\mathrm{h}}\|_{j,i-j}^2\big)
\nonumber\\[1.5mm]
&\lesssim
(\|u_{\mathrm{h}}\|^2_{\underline{j+1},i-j+1}+(i-j)\|\nabla p\|_{j+1,i-j-1}^2)+\|u_t\|_{i}^2+\|\mathcal{N}^1\|_{i}^2+\|(\mathcal{N}^2,\mathcal{N}^3)\|_{i+1}^2\nonumber\\[1.5mm]
&\quad\;+\|\nabla p\|_{i}(\|\nabla u_{\mathrm{h}}\|_{i+1}^{1/2}+\|\nabla u_{\mathrm{h}}\|_2^{1/2})\|\nabla u_{\mathrm{h}}\|_{\underline{i+1},0}^{1/2}.
\end{align}

{\bf{Step 2. Estimates for $(\bar{B}\cdot \nabla)^2\eta_{\mathrm{h}}$.}}
In analogy with the derivation of \eqref{2092650411635},
we apply $\partial_{\mathrm{h}}^{j+k}\partial_{3}^{i-j-k}$ to \eqref{2026195637es}$_1$ and multiply the resulting identity by $(\bar{B}\cdot \nabla)^2\partial_{\mathrm{h}}^{j+k}\partial_{3}^{i-j-k}\eta_{\mathrm{h}}$ in $L^2$, which yields
\begin{align}\label{2092650411338nmh}
&\|(\bar{B}\cdot \nabla)^2\partial_{\mathrm{h}}^{j+k}\partial_{3}^{i-j-k}\eta_{\mathrm{h}}\|_{0}^2
+(\mu+\chi)\int\partial_{\mathrm{h}}^{j+k}\partial_{3}^{2+i-j-k}u_{\mathrm{h}}\cdot(\bar{B}\cdot \nabla)^2\partial_{\mathrm{h}}^{j+k}\partial_{3}^{i-j-k}\eta_{\mathrm{h}}\mathrm{d}y
\nonumber\\[1mm]&
+2\chi\int\partial_{\mathrm{h}}^{j+k}\partial_{3}^{i-j-k}W_{\mathrm{h}}\cdot(\bar{B}\cdot \nabla)^2\partial_{\mathrm{h}}^{j+k}\partial_{3}^{i-j-k}\eta_{\mathrm{h}}\mathrm{d}y
=\int\partial_{\mathrm{h}}^{j+k}\partial_{3}^{i-j-k}\tilde{\mathcal{N}}_{\mathrm{h}}
\cdot(\bar{B}\cdot \nabla)^2\partial_{\mathrm{h}}^{j+k}\partial_{3}^{i-j-k}\eta_{\mathrm{h}}\mathrm{d}y.
\end{align}
Exploiting Young's inequality and summing over $k$, we obtain
\begin{align}\label{2092650411638nmh}
\|(\bar{B}\cdot \nabla)^2\eta_{\mathrm{h}}\|_{j,i-j}^2
&\lesssim\big(\|u_{\mathrm{h}}\|_{j,i-j+2}^2+\|W_{\mathrm{h}}\|_{j,i-j}^2
+\|u_t\|_{i}^2+\|\mathcal{N}^1\|_{i}^2\big)\nonumber\\[1mm]
&\quad+\sum_{k=0}^{i-j}\left|\int\partial_{\mathrm{h}}^{j+k}\partial_{3}^{i-j-k}\nabla_{\mathrm{h}} p
\cdot(\bar{B}\cdot \nabla)^2\partial_{\mathrm{h}}^{j+k}\partial_{3}^{i-j-k}\eta_{\mathrm{h}}\mathrm{d}y\right|:=I_{12}.
\end{align}
We now estimate the term $I_{12}$.
Specifically, for the case $i=0$, we use the interpolation inequality to obtain that
\begin{align*}
I_{12}\lesssim\|\nabla p\|_{0}\|(\bar{B}\cdot \nabla)\eta_{\mathrm{h}}\|_{1}\lesssim\|\nabla p\|_{0}\|(\bar{B}\cdot\nabla)\eta_{\mathrm{h}}\|_{2}^{1/2}\|(\bar{B}\cdot \nabla)\eta_{\mathrm{h}}\|_{0}^{1/2}.
\end{align*}
On the other hand, for the case $i\neq0$, we follow the derivation of \eqref{202604261537c} to get
\begin{align*}
I_{12}\lesssim&
(i-j)\|\nabla p\|_{j+1,i-j-1}\|(\bar{B}\cdot \nabla)^2\eta_{\mathrm{h}}\|_{j,i-j}\nonumber\\[1mm]
&+\|\nabla p\|_{i}(\|(\bar{B}\cdot \nabla)\eta_{\mathrm{h}}\|_{i+1}^{1/2}+\|(\bar{B}\cdot \nabla)\eta_{\mathrm{h}}\|_{2}^{1/2})\|\|(\bar{B}\cdot \nabla)\eta_{\mathrm{h}}\|_{\underline{i+1},0}^{1/2}.
\end{align*}
Finally, substituting the above two estimates into \eqref{2092650411638nmh} and using Young's inequality again, we arrive at
\begin{align}\label{2092650411738nmh}
\|(\bar{B}\cdot \nabla)^2\eta_{\mathrm{h}}\|_{j,i-j}^2
\lesssim&\big(\|u_{\mathrm{h}}\|_{j,i-j+2}^2+(i-j)\|\nabla p\|_{j+1,i-j+1}^2
+\|W_{\mathrm{h}}\|_{j,i-j}^2+\|u_t\|_{i}^2+\|\mathcal{N}^1\|_{i}^2\big)\nonumber\\[1mm]
&+\|\nabla p\|_{i}(\|(\bar{B}\cdot \nabla)\eta_{\mathrm{h}}\|_{i+1}^{1/2}+\|(\bar{B}\cdot \nabla)\eta_{\mathrm{h}}\|_{2}^{1/2})\|\|(\bar{B}\cdot \nabla)\eta_{\mathrm{h}}\|_{\underline{i+1},0}^{1/2}.
\end{align}

{\bf{Step 3. Estimate for $W_3$.}}
Applying $\partial_{\mathrm{h}}^{j+k}\partial_{3}^{i-j-k}$ to \eqref{2026195638es}, and taking the inner product of the resulting
identity with $\partial_{\mathrm{h}}^{j+k}\partial_{3}^{i-j-k}W_3$ in $L^2$, we obtain
\begin{align*}
&\frac{1}{2}\frac{\mathrm{d}}{\mathrm{d}t}\|\partial_{\mathrm{h}}^{j+k}\partial_{3}^{i-j-k} W_{3}\|_{0}^2
+4\chi\|\partial_{\mathrm{h}}^{j+k}\partial_{3}^{i-j-k}W_3\|_{3}^2
=\int\partial_{\mathrm{h}}^{j+k}\partial_{3}^{i-j-k}\tilde{\mathcal{M}}_{3}\cdot\partial_{\mathrm{h}}^{j+k}\partial_{3}^{i-j-k}W_{3}\mathrm{d}y.
\end{align*}
From which we get that there exists a positive constant $c_2>0$ such that
\begin{align}\label{202604262029}
&\frac{1}{2}\frac{\mathrm{d}}{\mathrm{d}t}\| W_{3}\|_{j,i-j}^2+c_2\|W_3\|_{j,i-j}^2
\lesssim\big(\|u\|_{j+1,i-j+1}^2+\|\mathcal{N}^2\|_{i+1}^2\big).
\end{align}

{\bf{Step 4. Conclusion.}}
We now conclude the proof.
Since $\partial_3\eta_3=\mathrm{div}\eta-\mathrm{div}_{\mathrm{h}}\eta_{\mathrm{h}}$, it is easy to check that
\begin{align}\label{2092650411339}
\|(\bar{B}\cdot \nabla)^2\eta_3\|_{j,i-j}
\lesssim\|(\bar{B}\cdot \nabla)^2\eta_{\mathrm{h}}\|_{j,i-j}+\|(\bar{B}\cdot \nabla)\eta\|_{\underline{i+1},0}+\|(\bar{B}\cdot \nabla)\mathrm{div}\eta\|_{i}.
\end{align}
Similarly, it holds that
\begin{align}\label{209265041136}
\|u_3\|_{j,i-j+2}
\lesssim\|u_{\mathrm{h}}\|_{j,i-j+2}+\|\nabla u\|_{\underline{i+1},0}+\|\mathrm{div}u\|_{i+1}.\qquad
\end{align}
Combining the above two estimates with \eqref{2092650411635}, \eqref{2092650411738nmh} and \eqref{202604262029}, it follows that there exists a positive constant $c_3>0$, such that
\begin{align}\label{202604262036}
&\frac{1}{2}\frac{\mathrm{d}}{\mathrm{d}t}\left(\lambda\|(\bar{B}\cdot \nabla)^2\eta_{\mathrm{h}}\|_{j,i-j}^2
+\bar{B}_3^2\| W_{\mathrm{h}}\|_{j,i-j}^2+\| W_3\|_{j,i-j}^2\right)
+c_3\big(\|u\|_{j,2+i-j}^2+\|((\bar{B}\cdot \nabla)^2\eta,W)\|_{j,i-j}^2\big)
\nonumber\\[1mm]
&\lesssim
(\|u\|_{\underline{j+1},i-j+1}^2+(i-j)\|\nabla p\|_{j+1,i-j-1}^2)
+\|u_t\|_{i}^2+\|\mathcal{N}^1\|_{i}^2+\|(\mathrm{div}\eta,\mathcal{N}^2,\mathcal{N}^3)\|_{i+1}^2\nonumber\\[1mm]
&\quad+\|\nabla p\|_{i}\big(\|(\nabla u_{\mathrm{h}},(\bar{B}\cdot \nabla)\eta_{\mathrm{h}})\|_{i+1}^{1/2}
+\|(\nabla u_{\mathrm{h}},(\bar{B}\cdot \nabla)\eta_{\mathrm{h}})\|_{2}^{1/2}\big)
\|(\nabla u_{\mathrm{h}},(\bar{B}\cdot \nabla)\eta_{\mathrm{h}})\|_{\underline{i+1},0}^{1/2}\nonumber\\[1mm]
&\quad
+\|(\nabla u,(\bar{B}\cdot \nabla)\eta)\|_{\underline{i+1},0}^2.
\end{align}

For the case where $j<i$.
By substituting $\partial_t^{j}$ with $\partial_\mathrm{h}^{j+1}$ in \eqref{n0101nn928nes}
and invoking the regularity theory for the Stokes problem (see Lemma \ref{10220835}), analogously to \eqref{112922471039},  we can have
\begin{align*}
&\|u\|_{j+1,i-j+1}^2+\|\nabla p\|_{j+1,i-j-1}^2\nonumber\\[1mm]
&\lesssim\|u\|_{j+1,0}^2+\|(u_t,(\bar{B}\cdot \nabla)^2\eta, W)\|_{j+1,i-j-1}^2+\|\mathcal{N}^1\|_{i}^2+\|\mathcal{N}^3\|_{i+1}^2.
\end{align*}
Consequently, for the case $j<i$, we refine \eqref{202604262036} to be
\begin{align}\label{202604262051}
&\frac{1}{2}\frac{\mathrm{d}}{\mathrm{d}t}\left(\lambda\|(\bar{B}\cdot \nabla)^2\eta_{\mathrm{h}}\|_{j,i-j}^2
+\bar{B}_3^2\| W_{\mathrm{h}}\|_{j,i-j}^2+\| W_3\|_{j,i-j}^2\right)
+c_3\big(\|u\|_{j,i-j+2}^2+\|((\bar{B}\cdot \nabla)^2\eta,W)\|_{j,i-j}^2\big)
\nonumber\\[1mm]
&\lesssim
\|u\|_{j+1,i-j+1}^2+\|((\bar{B}\cdot \nabla)^2\eta,W)\|_{j+1,i-j-1}^2
+\|u_t\|_{i}^2+\|\mathcal{N}^1\|_{i}^2+\|(\mathrm{div}\eta,\mathcal{N}^2,\mathcal{N}^3)\|_{i+1}^2\nonumber\\[1mm]
&\quad+\|\nabla p\|_{i}\big(\|(\nabla u_{\mathrm{h}},(\bar{B}\cdot \nabla)\eta_{\mathrm{h}})\|_{i+1}^{1/2}
+\|(\nabla u_{\mathrm{h}},(\bar{B}\cdot \nabla)\eta_{\mathrm{h}})\|_{2}^{1/2}\big)
\|(\nabla u_{\mathrm{h}},(\bar{B}\cdot \nabla)\eta_{\mathrm{h}})\|_{\underline{i+1},0}^{1/2}\nonumber\\[1mm]
&\quad+\|(\nabla u,(\bar{B}\cdot \nabla)\eta)\|_{\underline{i+1},0}^2+\|u\|_{i+1}^2.
\end{align}

On the other hand, in the case $i=j$, estimate \eqref{202604262036} reduces to
\begin{align}\label{202604262036nmnm}
&\frac{1}{2}\frac{\mathrm{d}}{\mathrm{d}t}\left(\lambda\|(\bar{B}\cdot \nabla)^2\eta_{\mathrm{h}}\|_{i,0}^2
+\bar{B}_3^2\| W_{\mathrm{h}}\|_{i,0}^2+\| W_3\|_{i,0}^2\right)
+c_3\big(\|u\|_{i,2}^2+\|((\bar{B}\cdot \nabla)^2\eta,W)\|_{i,0}^2\big)
\nonumber\\[1mm]
&\lesssim\|u\|_{\underline{i+1},1}^2+\|u_t\|_{i}^2+\|\mathcal{N}^1\|_{i}^2+\|(\mathrm{div}\eta,\mathcal{N}^2,\mathcal{N}^3)\|_{i+1}^2\nonumber\\[1mm]
&\quad+\|\nabla p\|_{i}\big(\|(\nabla u_{\mathrm{h}},(\bar{B}\cdot \nabla)\eta_{\mathrm{h}})\|_{i+1}^{1/2}
+\|(\nabla u_{\mathrm{h}},(\bar{B}\cdot \nabla)\eta_{\mathrm{h}})\|_{2}^{1/2}\big)
\|(\nabla u_{\mathrm{h}},(\bar{B}\cdot \nabla)\eta_{\mathrm{h}})\|_{\underline{i+1},0}^{1/2}\nonumber\\[1mm]
&\quad
+\|(\nabla u,(\bar{B}\cdot \nabla)\eta)\|_{\underline{i+1},0}^2.
\end{align}
Therefore, in view of the recursive inequality \eqref{202604262051} on $j$
and \eqref{202604262036nmnm}, along with the interpolation and Friedrichs inequalities applied to $\|u\|_{i+1}^2$ and $\|u\|_{\underline{i+1},1}^2$,
we conclude that there exist positive constants $h_{i,j}$ such that
\begin{align}\label{202604262051n}
&\frac{1}{2}\frac{\mathrm{d}}{\mathrm{d}t}\mathcal{H}_{i}(\eta,W)
+c_3\big(\|u\|_{i+2}^2+\|((\bar{B}\cdot \nabla)^2\eta,W)\|_{i}^2\big)
\nonumber\\[1mm]
&\lesssim
\|(\nabla u,(\bar{B}\cdot \nabla)\eta)\|_{\underline{i+1},0}^2+\|u_t\|_{i}^2+\|\mathcal{N}^1\|_{i}^2+\|(\mathrm{div}\eta,\mathcal{N}^2,\mathcal{N}^3)\|_{i+1}^2\nonumber\\[1mm]
&\quad+\|\nabla p\|_{i}\big(\|(\nabla u_{\mathrm{h}},(\bar{B}\cdot \nabla)\eta_{\mathrm{h}})\|_{i+1}^{1/2}
+\|(\nabla u_{\mathrm{h}},(\bar{B}\cdot \nabla)\eta_{\mathrm{h}})\|_{2}^{1/2}\big)
\|(\nabla u_{\mathrm{h}},(\bar{B}\cdot \nabla)\eta_{\mathrm{h}})\|_{\underline{i+1},0}^{1/2},
\end{align}
where $\mathcal{H}_{i}(\eta,W)$ is defined by
\begin{align*}
\mathcal{H}_{i}(\eta,W):=\sum_{j=0}^{i}h_{i,j}\left(\lambda\|(\bar{B}\cdot \nabla)^2\eta_{\mathrm{h}}\|_{j,i-j}^2
+\bar{B}_3^2\| W_{\mathrm{h}}\|_{j,i-j}^2+\| W_3\|_{j,i-j}^2\right).
\end{align*}
It is clear that $\mathcal{H}_{i}(\eta,W)$ satisfies \eqref{202604261336}.
Finally, invoking the preliminary estimates \eqref{11202054} and \eqref{06011711jumpv}--\eqref{11210836}, we immediately arrive at \eqref{202604102028}.
This completes the proof.
\hfill$\Box$
\end{pf}

Next we turn to the estimate for $\mathrm{div}w$.
Following the same procedure as in \eqref{202604101338}, by applying the divergence operator to \eqref{202609221247n}$_3$ and using the identity $\mathrm{div}\nabla\times u=0$,
we obtain
\begin{align}\label{202604101338nm}
\partial_t\mathrm{div}w+4\chi\mathrm{div}w=\mathrm{div}\mathcal{N}^{2}.
\end{align}
We have the following estimates for $\mathrm{div}w$.
\begin{lem}\label{divwnormal}
Under assumption \eqref{aprpioses} with sufficiently small $\delta$, it holds that, 
\begin{align}\label{20260411nd}
&\frac{\mathrm{d}}{\mathrm{d}t}\|\mathrm{div}w\|_{i}^2+c\|\mathrm{div}w\|_{i}^2
\lesssim
\begin{cases}
\sqrt{\mathcal{E}_{H}}{\mathcal{D}}_{L}
& \;\;\mathrm{ for }\;\; i=1;\\[1mm]
\sqrt{\mathcal{E}_{H}}{\mathcal{D}}_{H}+\|(\eta,u)\|_{3}\|\eta\|_{6}^2
& \;\;\mathrm{ for }\;\; i=4.
\end{cases}
\end{align}
\end{lem}
\begin{pf}
Taking the norm $\|\cdot\|_i$ of both sides of \eqref{202604101338nm}, we can obtain that
\begin{align}\label{20260411ne}
&2\chi\frac{\mathrm{d}}{\mathrm{d}t}\|\mathrm{div}w\|_{i}^2+\|4\chi\mathrm{div}w\|_{i}^2+\|\partial_t\mathrm{div}w\|_{i}^2
=\|\mathrm{div}\mathcal{N}^{2}\|_{i}^2.
\end{align}
Thanks to \eqref{11210836}, we can estimate that
$$\|\mathrm{div}\mathcal{N}^{2}\|_{i}^2\lesssim\|\mathcal{N}^{2}\|_{i+1}^2
\lesssim
\begin{cases}
\sqrt{\mathcal{E}_{H}}{\mathcal{D}}_{L}
& \;\hbox{ for }\; i=1;\\[1mm]
\sqrt{\mathcal{E}_{H}}{\mathcal{D}}_{H}+\|(\eta,u)\|_{3}\|\eta\|_{6}^2
& \;\hbox{ for }\; i=4.
\end{cases}$$
Putting it into \eqref{20260411ne} then yields \eqref{20260411nd}. This completes the proof.
\hfill$\Box$
\end{pf}

\section{Proof of Theorem \ref{thm1}} \label{2025thm1}

\subsection{Energy inequality}
With the energy estimates established in Section \ref{energy estimates},
we now are ready to build the lower-order, higher-order and highest-order energy inequalities.

\begin{pro}
\label{pro:0501nm}
Under assumption \eqref{aprpioses} with sufficiently small $\delta$, there exist energy functionals $\tilde{\mathcal{E}}_{L}$, $\tilde{\mathcal{E}}_{H}$ and $\overline{\|\eta\|}^2_6$
which are equivalent to $\mathcal{E}_{L}$, $\mathcal{E}_{H}$ and $\|\eta\|_6^2$, respectively, such that
\begin{align}
&\label{202108061542}
\frac{\mathrm{d}}{\mathrm{d}t} \tilde{\mathcal{E}}_{L}+\mathcal{D}_{L}\leqslant0,\\[1mm]
&\label{202108061042}
\frac{\mathrm{d}}{\mathrm{d}t} \tilde{\mathcal{E}}_{H}+\mathcal{D}_{H}\lesssim\sqrt{\mathcal{D}_{L}}\|\eta\|_{6}^2,\\[1mm]
&\label{202108061043}
\frac{\mathrm{d}}{\mathrm{d}t}\overline{\|\eta\|}^2_6+\|\eta\|_{6}^2
\lesssim{\mathcal{E}}_{H}+\mathcal{D}_{H} \quad\mathrm{ on }\;\;(0,T].\qquad
\end{align}
\end{pro}
\begin{pf}
To begin with,  it follows from Lemmas \ref{lem:21092401}--\ref{badiseqinM} that there exists two suitably large constants $c_4$ and $c_5$ such that
\begin{align}
&\label{202109241633}
\frac{\mathrm{d}}{\mathrm{d}t}\bar{\mathcal{E}}_{L}+c\bar{\mathcal{D}}_{L}\lesssim\sqrt{\mathcal{E}_{H}}{\mathcal{D}}_{L},\\
&\label{202109241635}
\frac{\mathrm{d}}{\mathrm{d}t}\bar{\mathcal{E}}_{H}+c\bar{\mathcal{D}}_{H}\lesssim\|u_{t}\|_{2}^2+\sqrt{\mathcal{E}_{H}}{\mathcal{D}}_{H}
+(\|(\eta,u)\|_3+\|(\nabla w,\nabla p)\|_{1})\|\eta\|_{6}^2,
\end{align}
where
\begin{align*}
&\bar{\mathcal{E}}_{L}:=
\sum_{i=0}^3\left(2\int\partial_{\mathrm{h}}^{i}u\cdot\partial_{\mathrm{h}}^{i}\eta\mathrm{d}y
+\int\partial_{\mathrm{h}}^{i}w\cdot\partial_{\mathrm{h}}^{i}\nabla\times\eta\mathrm{d}y
+\mu\|\nabla\partial_\mathrm{h}^{i} \eta\|_0^2+\chi\|\mathrm{div}\partial_\mathrm{h}^{i} \eta\|_0^2\right)\\
&\qquad\;\;+c_4\sum_{i=0}^3\big(\|(\partial_\mathrm{h}^{i} u,\partial_\mathrm{h}^{i} w)\|^2_0+\lambda\|(\bar{B}\cdot\nabla)\partial_\mathrm{h}^{i}\eta\|_{0}^2\big)
+\left(\|(\partial_{t}u,\partial_{t}w)\|_{0}^2+\lambda\|(\bar{B}\cdot\nabla)u\|_{0}^2\right),\\[1mm]
&\bar{\mathcal{E}}_{H}:=
\sum_{i=0}^5\left(2\int\partial_{\mathrm{h}}^{i}u\cdot\partial_{\mathrm{h}}^{i}\eta\mathrm{d}y
+\int\partial_{\mathrm{h}}^{i}w\cdot\partial_{\mathrm{h}}^{i}\nabla\times\eta\mathrm{d}y
+\mu\|\nabla\partial_\mathrm{h}^{i} \eta\|_0^2+\chi\|\mathrm{div}\partial_\mathrm{h}^{i} \eta\|_0^2\right)\\
&\qquad\;\;+c_4\sum_{i=0}^5\big(\|(\partial_\mathrm{h}^{i} u,\partial_\mathrm{h}^{i} w)\|^2_0+\lambda\|(\bar{B}\cdot\nabla)\partial_\mathrm{h}^{i}\eta\|_{0}^2\big)
+c_5\sum_{j=1}^2\left(\|(\partial_{t}^ju,\partial_{t}^jw)\|_{0}^2+\lambda\|(\bar{B}\cdot\nabla)\partial_{t}^{j-1}u\|_{0}^2\right)\\[1mm]
&\qquad\;\;+\left((\mu+\chi)\|\nabla_{\mathcal{A}}\partial_t^2u\|_0^2+4\chi\|\partial_{t}^2w\|_0^2-4\chi\int\nabla_{\mathcal{A}}\times \partial_{t}^2u\cdot\partial_{t}^2w\mathrm{d}y\right)
\end{align*}
and
\begin{align*}
&\bar{\mathcal{D}}_{L}:=\|(\bar{B}\cdot\nabla)\eta\|_{\underline{3},0}^2+c_4\|(\nabla u,w)\|_{\underline{3},0}^2
+\|(\nabla_{\mathcal{A}}u_t,w_t)\|_0^2,\\[1mm]
&\bar{\mathcal{D}}_{H}:=\|(\bar{B}\cdot\nabla)\eta\|_{\underline{5},0}^2+c_4\|(\nabla u,w)\|_{\underline{5},0}^2
+\|(\partial_t^3u,\partial_t^3w)\|^2_{0}+c_5\sum_{j=1}^2\|(\nabla_{\mathcal{A}}\partial_{t}^ju,\partial_{t}^jw)\|_0^2.\qquad
\end{align*}

Moreover, combining with \eqref{202604271129}, \eqref{202604271303},  \eqref{202604102028}, and \eqref{20260411nd}, we deduce that there exist two suitably large constants $c_6$ and $c_7$ such that
\begin{align}
&\label{2021080615420501}
\frac{\mathrm{d}}{\mathrm{d}t} \mathcal{E}_{1}+c\mathcal{D}_{1}
\lesssim\|(\eta,u)\|_{\underline{2},1}^2+\|u_t\|_{1}^2+\|\nabla p\|_{1}\|(\eta,u)\|_{3}^{1/2}\|((\bar{B}\cdot \nabla)\eta,\nabla u)\|_{\underline{2},0}^{1/2}\nonumber\\[1mm]
&\qquad\qquad\qquad+\sqrt{\mathcal{E}_{H}}{\mathcal{D}}_{L},\\[1mm]
&\label{2021080610420501}
\frac{\mathrm{d}}{\mathrm{d}t} \mathcal{E}_{2}+c\mathcal{D}_{2}\lesssim\|(\eta,u)\|_{\underline{4},1}^2+\|u_{t}\|_{4}^2
+\|\nabla p\|_{4}\|((\bar{B}\cdot \nabla)\eta,\nabla u)\|_{5}^{1/2}\|((\bar{B}\cdot \nabla)\eta,\nabla u)\|_{\underline{5},0}^{1/2}\nonumber\\[1mm]
&\qquad\qquad\qquad
+\|((\bar{B}\cdot \nabla)\eta, \nabla u)\|_{\underline{5},0}^2+\sqrt{\mathcal{E}_{H}}{\mathcal{D}}_{H}+\sqrt{\mathcal{D}_{L}}\|\eta\|_{6}^2,
\end{align}
where
\begin{align*}
&\mathcal{E}_{1}:=\overline{\|\eta\|}_{3}^2+c_6\mathcal{H}_{1}(\eta,W)+\|\mathrm{div}w\|_{1}^2,\\[1mm]
&\mathcal{E}_{2}:=\overline{\|\eta\|}_{5}^2+c_7\mathcal{H}_{4}(\eta,W)+\|\mathrm{div}w\|_{4}^2,\\[1mm]
&\mathcal{D}_{1}:=\|(\eta,u)\|_{3}^2+\|\nabla p\|_{1}^2+c_6\|((\bar{B}\cdot \nabla)^2\eta, W)\|_{1}^2+\|\mathrm{div}w\|_{1}^2,\\[1mm]
&\mathcal{D}_{2}:=\|\eta\|_{5}^2+\|u\|_{6}^2+\|\nabla p\|_{4}^2+c_7\|((\bar{B}\cdot \nabla)^2\eta, W)\|_{4}^2+\|\mathrm{div}w\|_{4}^2.
\end{align*}

Thanks to \eqref{poincaretype}, it is easy to calculate that
\begin{align}
&\label{202605012250}
\|(\bar{B}\cdot \nabla)\eta\|_{5}\lesssim\|\eta\|_{5}+\|(\bar{B}\cdot \nabla)^2\eta\|_{4}+\|(\bar{B}\cdot \nabla)\eta\|_{\underline{5},0},\\[1mm]
&\label{202605012250n}
\|\eta\|_{\underline{i},1}\lesssim\|(\bar{B}\cdot \nabla)\eta\|_{\underline{i+1},0}\quad\mbox{for any}\;i\geqslant0.
\end{align}
Hence, exploiting the interpolation inequality and \eqref{202605012250}--\eqref{202605012250n},
we can deduce from \eqref{202604271347}, \eqref{2026042712136} (with $k=1$), and \eqref{2021080610420501} that
there exists a suitably large constant $c_8$ such that
\begin{align}
&\label{202605012247}
\frac{\mathrm{d}}{\mathrm{d}t} \mathcal{E}_{2}+c\tilde{\mathcal{D}}_{2}\lesssim\sum_{j=0}^{3}\|\partial_t^{j}u\|_{0}^2
+\|\nabla p\|_{4}\|((\bar{B}\cdot \nabla)\eta,\nabla u)\|_{5}^{1/2}\|((\bar{B}\cdot \nabla)\eta,\nabla u)\|_{\underline{5},0}^{1/2}\nonumber\\[1mm]
&\qquad\qquad\qquad
+\|((\bar{B}\cdot \nabla)\eta, \nabla u)\|_{\underline{5},0}^2+\sqrt{\mathcal{E}_{H}}{\mathcal{D}}_{H}+\sqrt{\mathcal{D}_{L}}\|\eta\|_{6}^2,
\end{align}
where $\tilde{\mathcal{D}}_{2}$ is given by
\begin{align*}
\tilde{\mathcal{D}}_{2}:=\mathcal{D}_{2}+\sum_{j=1}^{2}\|\partial_t^jw\|_{7-2j}^2+c_8\sum_{j=1}^{2}\left(\|\partial_t^ju\|_{6-2j}^2+\|\nabla\partial_t^j p\|_{4-2j}^2\right)+\|(\bar{B}\cdot \nabla)\eta\|_{5}^2.
\end{align*}
Consequently, choosing $c_5$ to be sufficiently large,
applying \eqref{20190614fdsa1957}, \eqref{202605012250n}, Young's and interpolation inequalities, we can further deduce from \eqref{202109241633}--\eqref{2021080615420501} and \eqref{202605012247} that
there exist two suitably large constants $c_9$ and $c_{10}$ such that
\begin{align}
&\label{2021080615420501nm}
\frac{\mathrm{d}}{\mathrm{d}t} \mathcal{E}_{3}+c\mathcal{D}_{3}
\lesssim\sqrt{\mathcal{E}_{H}}{\mathcal{D}}_{L},\\[1mm]
&\label{2021080610420501nm}
\frac{\mathrm{d}}{\mathrm{d}t} \mathcal{E}_{4}+c\mathcal{D}_{4}\lesssim\sqrt{\mathcal{E}_{H}}{\mathcal{D}}_{H}+\sqrt{\mathcal{D}_{L}}\|\eta\|_{6}^2,
\end{align}
where
\begin{align*}
&\mathcal{E}_{3}:=\mathcal{E}_1+c_9\bar{\mathcal{E}}_{L},\quad\mathcal{D}_{3}:=\mathcal{D}_{1}+c_9\bar{\mathcal{D}}_{L}\\[1mm]
&\mathcal{E}_{4}:=\mathcal{E}_2+c_{10}\bar{\mathcal{E}}_{H},\quad\mathcal{D}_{4}:=\tilde{\mathcal{D}}_{2}+c_{10}\bar{\mathcal{D}}_{H}.
\end{align*}
By virtue of \eqref{20190614fdsa1957}, \eqref{202209210921nn}, \eqref{202604271630}, \eqref{202604271340}, \eqref{202604271345}, \eqref{2026042712136} (with $k=0$), \eqref{friedrich}--\eqref{poincaretype}, and the fact that
$$\|\nabla w\|_{i}\lesssim\|\partial_{\mathrm{h}} w\|_{i}+\|\partial_3 w\|_{i}
\lesssim\|w\|_{\underline{i+1},0}+\|\nabla\times w\|_{i}+\|\mathrm{div}w\|_{i}\;\;\mbox{for}\;i\geqslant0,$$
it is easily checked that under the assumption \eqref{aprpioses} with sufficiently small $\delta$, the functionals
$\mathcal{E}_{3}$, $\mathcal{E}_{4}$
are equivalent to $\mathcal{E}_L$, $\mathcal{E}_H$, and $\mathcal{D}_{3}$, $\mathcal{D}_{4}$ are equivalent to $\mathcal{D}_L$, $\mathcal{D}_H$, respectively.
Consequently, the inequalities \eqref{202108061542}--\eqref{202108061042} are direct consequences of \eqref{2021080615420501nm}--\eqref{2021080610420501nm}.
Additionally, the inequality \eqref{202108061043} directly follows from \eqref{202604271129} with $i=4$.
This completes the proof of Proposition \ref{pro:0501nm}.
\hfill$\Box$
\end{pf}

\subsection{Equivalence form of $\mathcal{E}_{H}$}
\begin{lem}\label{lem:12281620}
Under assumption \eqref{aprpioses} with sufficiently small $\delta$, it holds that
\begin{align}\label{12281622}
\mathcal{E}_{H}\;\;\mbox{is equivalent to}\;\;\|\eta\|_{\underline{5},1}^2+\|((\bar{B}\cdot\nabla)\eta, \eta,u,w)\|_5^2.
\end{align}
\end{lem}
\begin{pf}
Please refer to \cite[Lemma 3.6]{JFJSJMFMOSERT} for the proof.
\hfill$\Box$
\end{pf}

\subsection{A priori estimate}
With Proposition \ref{pro:0501nm} and Lemma \ref{lem:12281620} in hand,
we are now in a position to establish the \emph{a priori} estimate.
Making use of \eqref{202108061043}, we find that
$$\begin{aligned}
\|\eta\|_6^2&\lesssim\|\eta^0\|_6^2e^{-t}+\int_0^te^{-(t-\tau)}\left(\mathcal{E}_{H}+\mathcal{D}_{H}\right)(\tau)\mathrm{d}\tau\\
&\lesssim\|\eta^0\|_6^2e^{-t}+\sup_{\tau\in[0, t]} \mathcal{E}_{H}(\tau)\int_0^te^{-(t-\tau)}\mathrm{d}\tau
+\int_0^t\mathcal{D}_{H}(\tau)\mathrm{d}\tau\\[1mm]
&\lesssim\|\eta^0\|_6^2e^{-t}+\mathcal{G}_3(t),
\end{aligned}$$
which gives
\begin{align}\label{202108071121}
\mathcal{G}_1(t)\lesssim\|\eta^0\|_6^2+\mathcal{G}_3(t).
\end{align}

Multiplying \eqref{202108061043} by $(1+t)^{-3/2}$, we obtain
$$
\frac{\mathrm{d}}{\mathrm{d}t}\frac{\overline{\|\eta\|}_6^2}{(1+t)^{3/2}}+\frac{3}{2}\frac{\overline{\|\eta\|}_6^2}{(1+t)^{5/2}}
+\frac{\|\eta\|_6^2}{(1+t)^{3/2}}
\lesssim\frac{\mathcal{E}_{H}}{(1+t)^{3/2}}+\frac{\mathcal{D}_{H}}{(1+t)^{3/2}},
$$
which yields
\begin{align}\label{202108071132}
\mathcal{G}_2\lesssim\|\eta^0\|_6^2+\mathcal{G}_3(t).
\end{align}

On the other hand, an integration of \eqref{202108061042} with respect to time $t$ gives  rise to
\begin{align}\label{202309121114}
\mathcal{G}_3(t)&\lesssim\mathcal{E}_{H}(0)+\int_0^{t}\sqrt{\mathcal{D}_{L}}\|\eta(\tau)\|_{6}^2\mathrm{d}\tau,\nonumber\\[1mm]
&\lesssim\mathcal{E}_{H}(0)
+\mathcal{G}_{1}(t)\left(\int_0^t{(1+\tau)^{3/2}}{\mathcal{D}_{L}}\mathrm{d}\tau\right)^{1/2}\left(\int_0^t{(1+\tau)^{-3/2}}\mathrm{d}\tau\right)^{1/2}.
\end{align}

Let
$$\begin{aligned}
&\mathcal{G}_5(t):=\mathcal{G}_1(t)+\sup_{\tau\in[0, t]} \mathcal{E}_{H}(\tau)+\mathcal{G}_4(t).
\end{aligned}$$
From now on, we further assume $\sqrt{\mathcal{G}_5(T)}\leqslant \delta$, which is a stronger requirement than  \eqref{aprpioses}.
Then, we can use the above inequality and \eqref{12281622} with $t=0$ to obtain
\begin{align}\label{202108071147}
\mathcal{G}_3(t)&\lesssim\mathcal{E}_{H}(0)+\delta\mathcal{G}_{1}(t)
\lesssim\|\eta^0\|_{6}^2+\|(u^0,w^0)\|_{5}^2+\delta\mathcal{G}_{1}(t),
\end{align}
Hence, it follows from \eqref{202108071121}--\eqref{202108071147} that
\begin{align}\label{202110101533}
\sum_{i=1}^3\mathcal{G}_{i}(t)\lesssim\|\eta^0\|_{6}^2+\|(u^0,w^0)\|_{5}^2:=\mathcal{I}^0.
\end{align}

Finally, we show the time decay behavior of $\mathcal{G}_4(t)$.
Note that $\mathcal{E}_{L}$ can be controlled by $\mathcal{D}_{L}$
except for the term $\|\eta\|_{\underline{3},1}^2$ in $\mathcal{E}_{L}$.
However, we can use the interpolation inequality to obtain
$$\|\eta\|_4\lesssim\|\eta\|_3^{\frac{2}{3}}\|\eta\|_6^{\frac{1}{3}}.$$
Moreover, it follows from \eqref{202108071121} and \eqref{202108071147} that
$$\mathcal{E}_{L}+\|\eta\|_6^2\lesssim\tilde{\mathcal{E}}_{L}+\|\eta\|_6^2
\lesssim\mathcal{I}^0.$$
The combination of the above two inequalities gives rise to
$$\tilde{\mathcal{E}}_{L}\lesssim
\left(\mathcal{D}_{L}\right)^{\frac{2}{3}}\left(\mathcal{E}_{L}+\|\eta\|_6^2\right)^{\frac{1}{3}}
\lesssim\left(\mathcal{D}_{L}\right)^{\frac{2}{3}}\left(\mathcal{I}^0\right)^{\frac{1}{3}}.$$
Plugging the above estimate into \eqref{202108061542}, we obtain
$$\frac{\mathrm{d}}{\mathrm{d}t}\tilde{\mathcal{E}}_{L}+c
\frac{(\tilde{\mathcal{E}}_{L})^{3/2}}{(\mathcal{I}^0)^{1/2}}\leqslant0,
$$
which yields
$${\mathcal{E}}_{L}\lesssim\tilde{\mathcal{E}}_{L}
\lesssim\frac{{\mathcal{E}}_{L}(0)}{\left((\mathcal{I}^0/\mathcal{E}_{L}(0))^{1/2}+t/2\right)^2}
\lesssim\frac{{\mathcal{E}}_{L}(0)}{\left(1+t\right)^2}.$$
Consequently, it follows that
\begin{align}
\label{2021092616143}
&\sup_{0\leqslant\tau< t}(1+\tau)^{2}{\mathcal{E}}_{L}(\tau)\lesssim{\mathcal{E}}_{L}(0).
\end{align}

In addition, note that
$$\begin{aligned}
&\frac{\mathrm{d}}{\mathrm{d}t}\left((1+t)^{\frac{3}{2}}\tilde{\mathcal{E}}_{L}\right)
+(1+t)^{\frac{3}{2}}\mathcal{D}_{L}=(1+t)^{\frac{3}{2}}\left(\frac{\mathrm{d}}{\mathrm{d}t}\tilde{\mathcal{E}}_{L}
+\mathcal{D}_{L}\right)+\frac{3}{2}\tilde{\mathcal{E}}_{L}(1+t)^{\frac{1}{2}},
\end{aligned}$$
which, together with \eqref{202108061542} and \eqref{2021092616143}, leads to
\begin{align*}
&\frac{\mathrm{d}}{\mathrm{d}t}\left((1+t)^{\frac{3}{2}}\tilde{\mathcal{E}}_{L}\right)
+(1+t)^{\frac{3}{2}}\mathcal{D}_{L}
\lesssim\tilde{\mathcal{E}}_{L}(0)(1+t)^{-\frac{3}{2}}.
\end{align*}
Hence, we arrive at
\begin{align}
\label{202109261645}
\mathcal{G}_4(t)
\lesssim{\mathcal{E}}_{L}(0)\lesssim\mathcal{I}^0.
\end{align}

Summing up the above four estimates for $\mathcal{G}_{i}$, we conclude that
\begin{align}\label{202108071232}
\mathcal{G}(t):=\sum_{i=1}^4\mathcal{G}_{i}\lesssim
\left(\|\eta^0\|_6^2+\|(u^0,w^0)\|_5^2\right).
\end{align}
Consequently, by virtue of \eqref{202108071232},
we have proved the following \emph{a priori} estimate, which combines with the well-posedness result and a continuity argument,
yields Theorem \ref{thm1}.
\begin{pro}\label{pro:0807}
Let $(\eta,u,w)$ be a solution of the problem \eqref{202609221247} with an associated pressure $p$ (up to a constant).
Then there exists a sufficiently small $\delta$, such that $(\eta,u,w,p)$ enjoys the following estimate
\begin{align}\label{202604281950}
\mathcal{G}(T)\lesssim\left(\|\eta^0\|_6^2+\|(u^0,w^0)\|_5^2\right),
\end{align}
provided that $\sqrt{\mathcal{G}_5(T)}\leqslant \delta$ for some $T>0$.
\end{pro}

\vspace{3mm}
\noindent\textbf{Acknowledgements.}
The research of Youyi Zhao was supported by NSFC (12371233 and 12401289),
the Natural Science Foundation of Fujian Province of China (2024J08029),
and the Research Foundation of Fuzhou University (XRC-24050).

\vspace{5mm}
\noindent\textbf{Conflict of Interest.}
The author states that there is no conflict of interest.

\appendix
\section{Analysis tools}\label{sec:09}
\renewcommand\thesection{A}
This Appendix is aims to provide some mathematical analysis tools, which have been used in the previous sections.
It should be noted that in this appendix we still adapt the simplified mathematical notations in Section \ref{Main result}.
For the sake of simplicity, we still use the notation $a\lesssim b$ means that $a\leqslant cb$ for some constant $c>0$,
where the positive constant $c$ may
depend on the physical parameters, but dose not depend on the initial data or time, and may vary from line to line.

\begin{lem}\label{lem:10220826}
Embedding inequalities (see \cite[Theorem 4.12]{ARAJJFF}):
\begin{align}
&\label{embed237}
\|f\|_{L^p}\lesssim \| f\|_{1}\quad\mathrm{for}\;\;2\leqslant p\leqslant6,\\[1mm]
&\label{embed2}
\|f\|_{C^0(\overline{\Omega})}=\|f\|_{L^\infty}\lesssim \| f\|_{2}.
\end{align}
\end{lem}

\begin{lem}\label{10220830}
Friedrichs inequality (see \cite[Theorem 1.42]{NASII04}):
\begin{align}
\label{friedrich}
\|f\|_{0}\lesssim \|\nabla f\|_{0}\quad\mathrm{for}\;f\in H^1_0.
\end{align}
\end{lem}

\begin{lem}\label{10220830nm}
Poincar\'e-type inequality (see \cite[Lemma A.4]{WYTIVNMI}):
\begin{align}
\label{poincaretype}
\|f\|_{0}\lesssim \|(\bar{B}\cdot\nabla) f\|_{0}\quad\mathrm{for}\;f\in H^1_0.
\end{align}
\end{lem}

\begin{lem}\label{lem:10220837nm}
Young's inequality (see \cite{ELGP}): Let $1<p, q<\infty$, $\frac{1}{p} + \frac{1}{q} = 1$. Then
\begin{align}\label{young}
    ab \leqslant \frac{a^p}{p} + \frac{b^q}{q}\;\;\;(a,b>0).
\end{align}
A more general form, known as Young's inequality with $\epsilon$, states that 
\begin{align}\label{eyoung}
    ab \leqslant \epsilon a^p + C(\epsilon) b^q\;\;\;(a,b>0,\;\epsilon>0),
\end{align}
where $C(\epsilon)=(\epsilon p)^{-p/q} q^{-1}$.
\end{lem}
\begin{rem}
When $p = q = 2$, inequality \eqref{eyoung} reduces to the form
$$ab\leq \epsilon a^2+\frac{1}{4\epsilon}b^2\;\;\;(a,b>0,\;\epsilon>0),$$
which is widely used in energy estimates.
\end{rem}

\begin{lem}\label{10220828}
Interpolation inequality
(see \cite[Theorem 5.2]{ARAJJFF}):
For any $0\leqslant j< i$ and  $\epsilon>0$,
\begin{align}\label{interpolation}
&\|f\|_{j}\lesssim\|f\|_{0}^{1-{j}/{i}}\|f\|_{i}^{{j}/{i}}
\leqslant {C}( i,j,\epsilon)\|f\|_{0} +\epsilon\|f\|_{i}\quad\mathrm{for}\;f\in H^{i},
\end{align}
where the constant ${C}( i,j,\epsilon)$ depends on $i,j$ and $\epsilon$, and Young's inequality
has been used in the last inequality.
\end{lem}

\begin{lem}\label{10220833}
Product estimates (\cite[Lemma A.9]{JFJSWZhangwei}): Let $0\leqslant i\leqslant 6$, for $f,g\in H^{i}$, we have
\begin{align}
&\label{product}
 \|fg\|_{i}\lesssim    \begin{cases}
 \|f\|_{1}\|g\|_{1} & \mathrm{for}\;\;i=0;  \\[1mm]
 \|f\|_{i}\|g\|_{2} & \mathrm{for}\;\;0\leqslant i\leqslant 2\\[1mm]
                    \|f\|_{2}\|g\|_{j}
 +\|f\|_{j}\|g\|_{2}\qquad & \mathrm{for}\;\;3\leqslant i\leqslant 6.
                    \end{cases}
\end{align}
\end{lem}

\begin{lem}\label{lem:11190836}
Dual estimate (see \cite[Lemma A.8]{JFJHJS}): For $\varphi$, $\psi\in H^{1/2}(\mathbb{R}^2)$ and $\partial_{\mathrm{h}}\varphi\in L^1(\mathbb{R}^2)$,
\begin{align}\label{11190840}
\left|\int_{\mathbb{R}^2}\partial_{\mathrm{h}}\varphi\psi\mathrm{d}x_{\mathrm{h}}\right|
\lesssim|\varphi|_{1/2}|\psi|_{1/2}.
\end{align}
\end{lem}

\begin{lem}\label{10220835bnbm}
Trace estimate (see \cite[Lemma A.6]{JFJHJS}): For $f \in H^{1+i}$ with $i\geqslant0$, it holds that
\begin{equation}\label{37190928}
\| f|_{y_3=a} \|_{H^{i+1/2}(\mathbb{R}^2)} 
\lesssim \| f \|_{\underline{1+i},0}^{1/2} \| f \|_{\underline{i},1}^{1/2}\quad\mathrm{for\;any}\;a\in[0,1].
\end{equation}
\end{lem}

\begin{lem}\label{10220835}
Stokes estimate:
(see \cite[Lemma A.12]{WYJAIM2020}):
Let $k \geqslant 0$. Suppose that $f^2 \in H^{k}, f^3 \in H^{k+1}$, then $\varphi \in H^{k+2}, \nabla p \in H^{k}$ solving the problem
\begin{equation*}
\begin{cases}
\nabla \psi-\mu\Delta \varphi= f^{2}\quad &\mathrm{in}\;\;  \Omega,\\
\mathrm{div}\varphi=f^{3} &\mathrm{in}\;\;  \Omega,\\
\varphi=0 &\mathrm{on}\;\; \partial \Omega.
\end{cases}
\end{equation*}
Moreover, the solution enjoys that
\begin{align}\label{11190928}
\|\varphi\|_{k+2}+\|\nabla \psi\|_{k}\lesssim
\|\varphi\|_0^2 +\|f^{2}\|_{k}+\|f^{3}\|_{k+1}.
\end{align}
\end{lem}

\renewcommand\refname{References}
\renewenvironment{thebibliography}[1]{%
\section*{\refname}
\list{{\arabic{enumi}}}{\def\makelabel##1{\hss{##1}}\topsep=0mm
\parsep=0mm
\partopsep=0mm\itemsep=0mm
\labelsep=1ex\itemindent=0mm
\settowidth\labelwidth{\small[#1]}%
\leftmargin\labelwidth \advance\leftmargin\labelsep
\advance\leftmargin -\itemindent
\usecounter{enumi}}\small
\def\newblock{\ }
\sloppy\clubpenalty4000\widowpenalty4000
\sfcode`\.=1000\relax}{\endlist}
\bibliographystyle{model1b-num-names}

\end{document}